\documentclass{amsart}
\usepackage{graphicx}
\usepackage{amssymb,amscd,amsthm,amsxtra}
\usepackage{latexsym}
\usepackage{epsfig}
\usepackage{mathtools}
\usepackage{esint}
\usepackage{color}

\newtheorem{thm}{Theorem}[section]
\newtheorem{cor}[thm]{Corollary}
\newtheorem{lem}[thm]{Lemma}
\newtheorem{prop}[thm]{Proposition}
\theoremstyle{definition}
\newtheorem{defn}[thm]{Definition}
\theoremstyle{remark}
\newtheorem{rem}[thm]{Remark}
\numberwithin{equation}{section}
\newcommand{\R}{\mathbb R}

\newcommand{\eps}{\epsilon}

\newcommand{\p}{\partial}

\newcommand{\comment}[1]{}
\begin{document}

\title{Perturbative estimates for the one-phase Stefan Problem}
\author{D. De Silva}
\address{Department of Mathematics, Barnard College, Columbia University, New York, NY 10027}
\email{\tt  desilva@math.columbia.edu}
\author{N. Forcillo}
\address{Dipartimento di Matematica, Universit\`a di Bologna, Piazza di Porta San Donato, 5, 49126, Bologna, Italy}
\email{\tt nicolo.forcillo2@unibo.it}
\author{O. Savin}
\address{Department of Mathematics, Columbia University, New York, NY 10027}\email{\tt  savin@math.columbia.edu}
\thanks{N.F. is partially supported by INDAM-GNAMPA-2019 project: {\it Propriet\`a di regolarit\`a delle soluzioni viscose con applicazioni a problemi di frontiera libera,} and by the project: \textit{GHAIA Horizon 2020 MCSA RISE programme grant No 777822}}
%\thanks{N.F. wishes to thank the Department of Mathematics of Columbia University for the warm hospitality}
\begin{abstract} We provide perturbative estimates for the one-phase Stefan free boundary problem and obtain the regularity of flat free boundaries via a linearization technique in the spirit of the elliptic counterpart established in \cite{D}.
\end{abstract}

\maketitle

\section{Introduction}

In this paper we are concerned with perturbative estimates for the one-phase Stefan problem,
\begin{equation}\label{SP}\begin{cases}
u_t = \triangle u  &\text{in $(\Omega \times (0,T])  \cap\{u>0\},$}\\
u_t=|\nabla u|^2 &\text{on $(\Omega  \times (0,T] )\cap \p \{u>0\}$,}
\end{cases}\end{equation}
with $\Omega \subset \R^n$, $u:\Omega \times [0,T] \to \R$, $u \ge 0$.

The classical one-phase Stefan problem describes the phase transition between solids and liquids, such as the melting of the ice (see for example \cite{F}, \cite{R}). In this setting $u$ represents the temperature of the liquid, and the region $\{u = 0\}$ the unmelted region of ice. 

The main object of interest is the behavior of the free boundary $\p \{u > 0\}$. 
%Other contexts in which this problem arises are problems of flame propagation, and certain financial math problems, (see for example \cite{CV, LS}).
%Concerning the regularity of $u$, in general, a solution is not expected to be Lipschitz continuous (see for example \cite{CE}). However if the free boundary is Lipschitz in some space direction, then $u$ is Lipschitz continuous (which is clearly the optimal regularity,) as shown by  I. Athanasopoulos, L. Caffarelli, and S. Salsa in \cite{ACS1}. Thus, it remains to understand the regularity of the free boundary.
In problems of this type 
%though we are confronted with a ``hyperbolic" phenomenon i.e. 
free boundaries may not regularize instantaneously. A two dimensional example in which a Lipschitz free boundary preserves corners can be found for instance in \cite{CS}.  Athanasopoulos,  Caffarelli, and Salsa studied the regularizing properties of the free boundary under reasonable assumptions in the more general setting of the two-phase Stefan problem. In \cite{ACS1} they showed that Lipschitz free boundaries in space-time become smooth provided a nondegeneracy condition holds, while in \cite{ACS2} the same conclusion was established for sufficiently ``flat" free boundaries. 
%A non-degeneracy condition however is enough to guarantee the smoothing of a Lipschitz free boundary, as shown by I. Athanasopoulos, L. Caffarelli, and S. Salsa in their pioneer article \cite{ACS2}. The same techniques can be refined to determine that if the free boundary is ``flat'', i.e. it stays close to a Lipschitz graph in a space-time neighborhood, then the solution is indeed smooth. This is shown by the same authors in \cite{ACS3}. We remark that both \cite{ACS2,ACS3} deal with a two-phase type Stefan problem and that 
The techniques are based on the original work of Caffarelli in the elliptic case \cite{C1,C2}. 

A related result is due to S. Choi and I. Kim who showed in \cite{CK} that solutions regularize instantaneously if the initial free boundary is locally Lipschitz with bounded Lipschitz constant and the initial data has subquadratic growth.

In this paper we study the regularity of flat free boundaries for \eqref{SP} based on perturbation arguments leading to a linearization of the problem, which are in the spirit of the elliptic counterpart developed by the first author in \cite{D}. Our result is basically equivalent to the previously mentioned flatness result in \cite{ACS2}. The techniques in \cite{D} are very flexible and have been widely generalized to a variety of free boundary problems, including two-phase inhomogeneous problems, ``thin" free boundary problems, minimization problems (see for example \cite{DFS}, \cite{DR}, \cite{DSV}). The methods of the current paper are suitable to further extensions as well.

Our main theorem roughly states that a solution to the Stefan problem in a ball of size $\lambda$ in space-time which is of size $\lambda$ and has a ``flat free boundary" in space, must have smooth free boundary in the interior provided that a necessary nondegeneracy condition holds. The nondegeneracy condition for $u$ requires that $u$ is bounded below by a small multiple of $\lambda$ at some point in the domain at distance $\lambda$ from the free boundary. Precisely, we assume that $u:\Omega \times [0,T] \to \R^+$ solves \eqref{SP}
in the viscosity sense. This means that $u$ is continuous and its graph cannot be touched by above (resp. below) at a point $(x_0,t_0)$ in a 
parabolic cylinder $B_r(x_0) \times (t_0-r^2,t_0],$
by the graph  of a classical strict supersolution $\varphi^+$ (resp. subsolution). By a classical strict supersolution we mean that $\varphi(x,t) \in C^2$, $\nabla_x \varphi \ne 0$, and it solves \begin{equation}\label{Strict}\begin{cases}
\varphi_t > \triangle \varphi  & \text{in ($\Omega \times (0,T])  \cap\{\varphi>0\},$}\\
\varphi_t > |\nabla \varphi|^2 &\text{on $(\Omega \times (0,T]) \cap \p \{\varphi >0\}$.}
\end{cases}\end{equation} Similarly we can define a strict classical subsolution.

Throughout the paper, given a space-time function, $\nabla, \Delta,$ and $D^2$ are computed with respect to the space variable $x.$

The rigorous statement of the main theorem is as follows.

\begin{thm}\label{Main}
Fix a constant $K$ (large) and let $u$ be a solution to the one-phase Stefan problem \eqref{SP} in $B_{\lambda} \times [-K^{-1}\lambda,0]$ for some $\lambda \le 1$. 
Assume that
$$|u| \le K \lambda, \quad \quad u(x_0,t) \ge K^{-1} \lambda \quad \mbox{for some} \quad x_0 \in B_{\frac 34\lambda}.$$ 
There exists $\eps_0$ depending only on $K$ and $n$ such that if, for each $t$, $\p_x \{u>0\}$ is $\eps_0$-flat in $B_\lambda$, then the free boundary $\p \{u>0\}$ (and $u$ up to the free boundary) is smooth in $B_{\frac\lambda 2} \times [-(2K)^{-1}\lambda,0]$.
\end{thm}

Here we use the notation $\p_x \{u>0\}$ to denote the boundary in $\R^n$ of $\{u(\cdot, t)>0)\}$, with $t$ being fixed. By $\p_x \{u>0\}$ is $\eps_0$-flat in $B_\lambda$ we understand that, for each $t$, $\p_x \{u>0\} \cap B_\lambda$ is trapped in a strip 
of width $\eps_0 \lambda$ (the region between two parallel hyperplanes at distance $\eps_0 \lambda$ from each other), and $u=0$ on one side of this 
strip while $u>0$ on the other side.

The assumption that $u$ is of size $\lambda$ in a domain of size $\lambda$ around the free boundary is natural, since this 
eventually holds for all classical solutions by choosing $\lambda$ small. 
We point out that in Theorem \ref{Main} the behavior of the solution depends strongly on the value of $\lambda$. If we scale the domain to unit size and keep the function $u$ of size 1, then the rescaled function
$$ (x,t) \quad \mapsto \quad \frac 1 \lambda u(\lambda x,\lambda t) ,\quad \quad (x,t) \in B_1 \times [-K^{-1},0],$$
solves a Stefan problem with possibly large diffusion coefficient $\lambda^{-1}$
\begin{equation}\label{SPr}\begin{cases}
\lambda u_t = \triangle u  &\text{in $(B_1 \times (-K^{-1},0]) \cap\{u>0\},$}\\
u_t=|\nabla u|^2 &\text{on $(B_1  \times (-K^{-1},0])\cap \p \{u>0\}$}.
\end{cases}\end{equation}
Our theorem states that nondegenerate solutions of size 1 of \eqref{SPr} which have $\eps_0$- flat free boundaries in $B_1$ are smooth up to the free 
boundary. We remark that $\eps_0$ is independent of $\lambda$, which means that we need to obtain uniform estimates in $\lambda$ for the oscillation of the free 
boundaries of solutions of \eqref{SPr}. Our results show that the free boundary has a uniform $C^{1,\alpha}$ bound in space. On the other hand, the estimates for $u$ in the set where it is positive depend on the parameter $\lambda$. The strategy is to approximate $u$ with a family of explicit functions $l_{a,b}$ which in the direction perpendicular to the free boundary depend on $\lambda$ while on the tangential directions to the free boundary are independent of the parameter $\lambda$.

Formally as $\lambda \to 0^+$, a solution $u$ to \eqref{SPr} solves the Hele-Shaw equation. Estimates for this problem by similar methods as ours were obtained by H. Chang-Lara and N. Guillen in [CG].

To prove our main theorem, we show that if a solution $u$ satisfies the hypotheses of Theorem \ref{Main} then, after a convenient dilation, the flatness 
assumption can be extended to the whole function $u$ instead of just the free boundary. Then Theorem \ref{Main} follows from the following result.

\begin{thm}\label{Main2}
Fix a constant $K$ (large) and let $u$ be a solution to the one-phase Stefan problem \eqref{SP} in $B_{2\lambda} \times [-2\lambda,0]$ for some $\lambda \le 
1$. 
Assume that $0\in \p \{u>0\},$ and
 $$ a_n(t) \, \left (x_n - b(t)-\eps_1 \lambda \right)^+ \, \, \le u \, \, \le a_n(t) \, \left(x_n - b(t)+\eps_1 \lambda \right)^+,$$
 with
 $$K^{-1}\le a_n \le K, \quad \quad |a_n'(t)| \le \lambda^{-2}, \quad \quad b'(t)=-a_n(t),$$
for some small $\eps_1$ depending only on $K$ and $n$. Then in $B_\lambda \times[-\lambda ,0]$ the free boundary $\p \{u>0\}$ is a $C^{1,\alpha}$ graph in 
the $x_n$ direction.
\end{thm}

The assumption that $b'=-a_n(t)$ means that the approximating linear functions in $x$, $a_n(t)(x_n-b(t))^+$, satisfy the free boundary 
condition, while $|a_n'(t)| \le \lambda^{-2}$ respects the parabolic scaling of the interior equation and represents that $a_n$ can change at most $o(1)$ in a time interval of length $o(\lambda^2)$. 

We remark that it suffices to prove Theorem \ref{Main2} under the more relaxed hypotheses 
\begin{equation}\label{lala0}
\lambda \le \lambda_0 \quad \mbox{ and} \quad  |a_n'(t)| \le c_0 
\lambda^{-2},
\end{equation}
with $\lambda_0$, $c_0$ small depending on $K$, $n$. We end up in this setting by working in balls of size $\tau \lambda$ 
with $\tau$ sufficiently small, and then relabel $\tau \lambda$ by $\lambda$ and $\eps_1 \tau^{-1}$ by $\eps_1$. 

Theorem \ref{Main2} applies, for example, when $u$ is a perturbation of order $o(1)\lambda$ of a traveling wave solution
$$ (e^{a x_n + a^2 t} -1)^+, \quad \quad K^{-1} \le a \le K.$$
In this case we choose $a_n(t)=a$, $b(t)=-at$, and consider $\lambda \le \lambda_0$ small so that the difference between the approximating linear part 
$a_n(t)(x_n-b(t))$ and the exact solution above is less than $\frac 12 \eps_1 \lambda$ in $B_\lambda$.  

The proof of Theorem \ref{Main2} is based on linearization techniques. The linearized equation in our setting has the form of an oblique derivative parabolic problem
\begin{equation}\label{LPI}\begin{cases}
\lambda v_t = tr(A(t) D^2v) &\text{in $\{x_n>0\}$,}\\
v_t = \gamma(t) \cdot \nabla v &\text{on $\{x_n=0\}$,}
\end{cases}
\end{equation}
with $A(t)$ uniformly elliptic and $\gamma_n>0$. An important task in our analysis is to develop Schauder-type estimates for equation \eqref{LPI} with respect to an appropriate distance $d_\lambda$ and to capture 
both features of the mixed parabolic/hyperbolic scaling.

The paper is organized as follows. In the next section we show that Theorem \ref{Main} can be deduced from Theorem \ref{Main2}. In Section 3, we use a Hodograph transform to obtain an equivalent quasilinear parabolic equation with oblique derivative boundary condition. In the following section, we state an improvement of flatness result Proposition \ref{P!} for solutions of such nonlinear problem, then we show how this implies Theorem \ref{Main2}. The proof of Proposition \ref{P!} is presented in Section 5, and it relies on various H\"older estimates (with respect to the appropriate distance) for solutions to the linearized problem associated to the nonlinear problem. Sections 6 and 7 are devoted to the proofs of such H\"older estimates, while Section 8 focuses on the one dimensional linear problem, which plays an essential role. The last section contains some general technical results on solutions to the linear problem.

  \section{From flat free boundaries to flat solutions.}

In this section, we show that Theorem \ref{Main} can be reduced to Theorem \ref{Main2}.

We assume that the function $u$ satisfies the $\eps_0$-flatness hypothesis of the free boundary from Theorem \ref{Main} for some $\lambda \le 1$, 
and that $(0,0)$ is a free boundary point. 
Precisely, by $\p_x \{u>0\}$ is $\eps_0$-flat in $B_\lambda$ we understand that, for each $t$, there exists a direction $\nu$ such that $$\p_x \{u(\cdot, t)>0\} \cap B_\lambda \subset \{|(x-x_0) \cdot \nu| \leq \eps_0 \lambda\},$$
and 
\begin{align*}
&u=0 \quad \text{in $\{(x-x_0) \cdot \nu \leq -\eps_0 \lambda\},$}\\
&u>0 \quad \text{in $\{(x-x_0) \cdot \nu \geq \eps_0 \lambda\}$}.
\end{align*}

First, we show that in a smaller domain $B_{\eta \lambda} \times [- \eta \lambda,0]$ the whole graph of $u$ is $\eta^\beta$- flat, for some small $\beta$, provided that $
\eps_0 \le c(\eta,K)$. Then, in this domain the hypotheses of Theorem \ref{Main2} are satisfied by choosing $\eta$ sufficiently small.

We work with the parabolic rescaling of the function $u$ which is defined in $B_1 \times  [-(K\lambda)^{-1},0]$ and keeps the function $u$ of unit size: 
$$ (x,t) \quad \mapsto \quad \frac 1 \lambda u(\lambda x,\lambda^2 t) ,\quad \quad (x,t) \in B_1 \times  [-(K\lambda)^{-1},0].$$
By abuse of notation we denote this rescaling by $u$, and then $u$ solves a Stefan problem with possibly small speed coefficient $\lambda$,
\begin{equation}\label{SPr2}\begin{cases}
 u_t = \triangle u  & \text{in ($B_1 \times (-(K\lambda)^{-1},0]) \cap\{u>0\},$}\\
u_t=\lambda |\nabla u|^2 & \text{on $(B_1 \times (-(K\lambda)^{-1},0]) \cap  \p \{u>0\} $.}
\end{cases}\end{equation}

We prove the following main lemma. Universal constants only depend on $n,K$. As usual, in the body of the proofs, constants denoted by $C$ may change from line to line.

\begin{lem}\label{gfla}
Assume that $u$ solves \eqref{SPr2}, 
$$|u| \le K , \quad \quad u(x_0,t) \ge K^{-1}  \quad \mbox{for some} \quad x_0 \in B_{3/4},$$
$$0 \in \p_x\{u(\cdot,0)>0\}, \quad \quad \mbox{ and $\p_x\{u(\cdot,t)>0\}$ is $\eps_0$-flat in $B_1$.}$$
Then for all small $\eta>0$ we have
 $$a_n(t) \left(x_n -  b(t) -\eta^{1+\beta}\right) ^+ \, \le \, u \, \le \, a_n(t) \left(x_n -b(t) + \eta^{1+\beta} \right)^+ \quad \mbox{in }B_{\eta} \times [-\lambda^{-1}\eta,0],$$ 
with $\beta=1/20$ and for $c,C>0$ universal,
$$c \le a_n(t) \le C, \quad |a_n'(t)| \le \eta^{\beta-2}, \quad \quad b'(t)=- \lambda a_n(t), \quad b(0)=0,$$
provided that $\eps_0 \le c(\eta,K)$. 
\end{lem}

When we rescale the conclusion back to the original coordinates, we obtain that the hypotheses of Theorem \ref{Main2} are satisfied in the cylinder $B_{\eta \lambda} \times [-\eta \lambda,0]$ with $\eps_1=\eta^\beta$.

We start by proving a result about the location of the free boundary in time.

\begin{lem}\label{uso}
Assume $u$ solves \eqref{SPr2} in $B_2 \times [- K^{-1},1]$ and that $0\le u \le K $. If $u(x,0)=0$ in $B_1$, then
\begin{equation}\label{ulec}
u(x,t) \le C (|x|- 1)^+, \quad \mbox{if} \quad t \in [-(2K)^{-1},0],
\end{equation}
and
\begin{equation} \label{ulec2}
u(x,t)=0 \quad \mbox{ if $|x| <1-C \lambda$,\hspace{0.2cm}$t \in [0,1]$,} 
\end{equation}
with $C>0$ universal.
\end{lem}

\begin{proof} Since the support of $u$ is increasing with time we deduce that $u=0$ in $B_1$ for all $t \in [-K^{-1}, 0]$.
Then, in the annular domain $(B_2 \setminus B_1) \times [-K^{-1},0]$, by the comparison principle,  $u$ is less than a multiple of the solution to the heat equation which equals $0$ on $\p B_1 \times (-K^{-1}, 0]$, and $1$ on the remaining part of the parabolic boundary. This, together with the boundary regularity of such solution, implies the estimate \eqref{ulec}.

Now, for times $t \in [0,1]$ we compare $u$ with
$$w(x,t)=C_0 \, \, g(|x|-r(t)), \quad  r(t):=1 - C_0 \lambda t,$$
with $g$ a 1D function such that $g(s)=0$ if $s \le 0,$ and for positive $s$ is defined by the ODE
$$g''(s)+ 2n g'(s)=0, \quad g(0)=0, \quad g'(0)=1.$$
Notice that $g' \in [0,1]$. 

We may assume that $r(t) \ge 1/2$, otherwise the conclusion \eqref{ulec2} is trivial (say for $C >2C_0$). 

The constant $C_0$ is chosen large such that $w \ge u$ at time $t=0$ (by \eqref{ulec}) and also on $\p B_2 \times [0,1].$ We check that $w$ is a supersolution to \eqref{SPr2}; indeed in $\{w>0\}$ we have (recall $r(t) \geq 1/2$),
%Now the estimate \eqref{SPr2} follows since $w$ is supersolution to \eqref{SPr2} since
$$ w_t = C_0^2 \lambda g' \ge 0, \quad \quad \triangle w= C_0\left(g'' + \frac{n-1}{|x|} g'\right) <0,$$
and on $\p \{w>0\}$
$$ w_t=\lambda C_0^2 =\lambda |\nabla w|^2.$$
In conclusion, $u \le w$ which gives the desired conclusion \eqref{ulec2}. 
\end{proof}

Now, we turn to the proof of Lemma \ref{gfla}.

\begin{proof}[Proof of Lemma $\ref{gfla}$]
We assume that $u$ satisfies \eqref{SPr2} in $B_1 \times [-(K\lambda)^{-1},0]$, and $\p_x \{u>0\}$ is $\eps_0$-flat in $B_1$. 
Suppose that $(0,0) \in \p \{u>0\}$ and then, after a rotation, $$\mbox{$u(x,0)>0$ if $x_n > \eps_0$, and $u(x,0)=0$ if $x_n < - \eps_0$.}$$ 
From \eqref{ulec} in Lemma \ref{uso} (applied to balls tangent to $\{x_n=-\eps_0\}$) we find that $u \le C (x_n+ \eps_0)^+$ in $B_{1/2} \times [-(2K)^{-1},0]$. 

We define
$$u_\tau := \frac{1}{\tau} u(\tau x,\tau^2 t), \quad \quad \quad \mbox{with} \quad \tau \ge \eps_0^{1/2},$$
and, if $\tau \in [\eps_0^{1/2},c],$ then 
\begin{equation}\label{utau0}
u_\tau \le C(x_n+ \tau)^+ \quad \mbox{in $B_1 \times [-2,0]$.}
\end{equation} Notice that $u_\tau$ satisfies \eqref{SPr2} with $\tau \lambda$ instead of $\lambda$.
We apply \eqref{ulec2} of Lemma \ref{uso} for $u_\tau$ and obtain that (since $(0,0) \in \p\{u_\tau>0\}$),
\begin{equation}\label{flau}
\mbox{$\p_x \{u_\tau>0\} \cap B_{1/2}$ intersects $\{x_n \le C \lambda \tau\}$}, \quad \text{for all $t \in [-1,0].$}
\end{equation} 
Moreover,  $\p_x \{u_\tau>0\}$ is $\tau^{-1} \eps_0$-flat in $B_1$, which combined with \eqref{flau} implies that 
\begin{equation}\label{flau1}
\mbox{$\p \{u_\tau>0\} \cap (B_{1/2} \times [-1,0])$ is included in $\{x_n \le C (\lambda \tau + \tau^{-1}\eps_0)\}$.}
\end{equation}

In $(B_{1/2} \cap \{x_n > C \tau\}) \times [-1,0]$ we compare $u_\tau$ with the solution $w$ to the heat equation which equals $0$ 
on $\{x_n = C \tau\}$, and equals $u_\tau$ on the remaining part of the parabolic boundary. Notice that by \eqref{flau1}, since $\tau \geq \eps_0^{1/2}$, $u_\tau>0$ on $\{x_n=C\tau\}$.
From \eqref{utau0} we find $|u_\tau - w| \le C \tau$, and the boundary regularity of $w$ gives
\begin{equation}\label{ustau}
|u_\tau - ax_n| \le C \rho^{3/2} + C \tau \le 2C \rho^{3/2} \quad \quad \mbox {in} \quad B^+_{2\rho} \times[-\rho^2,0],
\end{equation}
for some constant $a<C$, provided that we choose $\tau = \rho^{3/2}$ with $\rho$ small, to be made precise later. 

We claim that the nondegeneracy assumption $u(x_0,t) \ge K^{-1}$ for some $x_0 \in B_{3/4}$ implies that $a>c$.
 For this we use \eqref{flau1} which,  in terms of the function $u$, implies that $\p_x \{u(\cdot,t)>0\}$,
 at all times $t=-\tau^2 \le -\eps_0$, intersects the $x_n$ axis at distance at most $C (\lambda |t|+\eps_0)$ from the origin.  As for \eqref{flau1},
using that $\p_x\{u>0\}$ is $\eps_0$-flat in $B_1,$ we obtain that $u(x,t)>0$ if $x_n > C \eps_0 + C \lambda |t|$ in $B_{1/2}$. Now we can use the nondegeneracy condition with a Hopf-type lemma for the heat equation and obtain $$ u \ge c (x_n-C(\eps_0+\lambda |t|))^+ \quad \mbox{in} \quad B_{1/4} \times [-(4K)^{-1}, 0],$$
for some $c>0$ that depends only on $n$ and $K$. We use this inequality at time $t=0$ in \eqref{ustau} and conclude $a>c$ since $\tau \rho >2 \tau^2 \ge 2\eps_0$. We can restate \eqref{ustau} as
$$ (ax_n - C \eta^{1+ \frac 1 5})^+ \le u \le  (ax_n + C \eta^{1+ \frac 15})^+ \quad \mbox{in} \quad B_{2 \eta} \times [-\eta^2,0],$$
with $\eta:=\tau \rho =\rho^{5/2}$.

Similarly, by looking at the points $(b(t)e_n,t)$ where the free boundary intersects the $x_n$ axis, we obtain that 
$$ |b(t)|\le C (\lambda |t|+\eps_0) \le C_0 \eta \quad \mbox{if} \quad t \in [-\lambda^{-1} \eta,0],$$
and in the domain $B_{2C_0\eta} \times [t-\eta^2,t]$ we have 
$$\left (a(t) \cdot (x - b(t)e_n) - C \eta^{\frac 65}\right)^+ \le u(x,s) \le \left (a(t) \cdot (x - b(t)e_n) + C \eta^{\frac 65}\right)^+ $$
with $c \le |a(t)| \le C$. 
The flatness assumption of the free boundary in $B_1$ implies  
$$|a(t)-a_n(t) e_n| \le C \eta,$$ so we may replace $a(t) \cdot (x-b(t) e_n)$ above by $a_n(t) (x_n-b(t))$. 

The bounds on $u$ above imply that $a_n(t)$ can vary at most $C \eta^{1/5}$ in an interval of length $\eta^2$. We can regularize $a_n(t)$ by 
averaging over such intervals (convolving with a mollifier) and the bounds for $u$ still hold after changing the value of the constant $C$. Hence for all $t \in [-\lambda^{-1}\eta,0]$, we can find $a_n(t) \in \R$ such that
\begin{equation}\label{e10}
a_n(t) \,  \left (x_n - b(t) - C \eta^{\frac 65} \right) ^+\le u \le a_n(t) \, \left(x_n -b(t) + C \eta^{\frac 65}\right)^+ 
\end{equation}
in $B_{2C_0 \eta} \times [t-\eta^2,t]$
with 
\begin{equation}\label{ansize}c \le a_n(t) \le C, \quad \quad |a_n'(t)| \le C \eta^{\frac 15-2}, \quad \quad |b(t)| \le C_0 \eta.\end{equation}
It remains to show that we can modify $b$ slightly so that it satisfies the ODE $b'=-\lambda a_n$. Precisely, we let 
$$ \tilde b'(t)= - \lambda a_n(t), \quad \tilde b(0)=0,$$
and we show that 
\begin{equation}\label{e11}
|b(t)-\tilde b(t)| \le C \eta^{1+\beta} \quad \mbox{ if} \quad t \in [-\lambda^{-1}\eta,0], \quad \quad \beta=1/10.
\end{equation}

For this we perturb the family of evolving planes $a_n(t)(x_n-\tilde b(t))^+$ into a subsolution/supersolution. Let
$$d(t):=\tilde b(t)+C_1 \eta^{\beta} \lambda t,$$
with $C_1$ large, to be specified later.
We claim that
\begin{equation}\label{e12}
b(t) \ge d(t) -  2\eta^{1+\beta}.
\end{equation}
For this we define the function
$$v:=(1-C_2\eta^{\beta}) \, a_n(t) \, (h(x-d(t)e_n))^+,$$
with
$$h(x):=x_n - \eta^{\beta-1} (|x'|^2-2n x_n^2),$$
and check that it is a subsolution to our problem \eqref{SPr2} in the domain $$\Omega:=\bigcup_{t \in [- \lambda^{-1}\eta,0]}B_{2\eta}(d(t)e_n) \times \{t\}.$$
 Notice that in a ball of radius $2\eta,$
 \begin{equation}\label{sizeh}h \le C \eta, \quad \quad |\nabla h|=1 + O(\eta^{\beta}),\end{equation}
 and the constant $C_2=C_2(n)$ is chosen depending only on $n$ such that
 \begin{equation}\label{e13}
v \le a_n(t) (x_n - d (t))^+,
 \end{equation}
 with equality at $d(t)e_n$ and moreover, when $x \in \p  B_{2\eta}(d(t)e_n) \cap \{v(x,t)>0\}$, the difference between the two functions above is greater than $\eta^{1+\beta}$.

Next, we check that $v$ is a subsolution. In the interior $\{v>0\}$, using $\eqref{ansize},\eqref{sizeh},$ the definition of $\tilde b$, we have (for $\eta$ small)
$$|v_t| \le C |a_n'|\eta + C|d'| \le C \eta^{-4/5}, \quad \quad \triangle v \ge  c\eta^{\beta-1} > v_t,$$
and on the free boundary ($C'$ depending only on $C_2,n$),
$$v_t = (1-C_2 \eta^{\beta})a_n (-d ') h_n, \quad \quad |\nabla v|^2 \ge (1-C' \eta^{\beta}) a_n^2 .$$
Since 
$$h_n=1+ O(\eta^{\beta}), \quad \quad \quad (-d')a_n = \lambda a_n^2 - C_1 \lambda a_n \eta^{\beta},$$
we can choose $C_1$ large such that $v_t < \lambda |\nabla v|^2$.

If $$\mbox{$b (t_0)< d(t_0) - 2\eta^{1+\beta}$ for some $t_0 \in [-\lambda^{-1}\eta,0]$, }$$
then by \eqref{e10} and \eqref{e13} we find that $v< u$ at time $t=t_0$ in $B_{2\eta}(d(t_0)e_n) \cap \overline {\{v >0\}}$. 
On the other hand $v=u$ at the origin $(0,0)$. 
This means that as we increase $t$ from $t_0$ to $0$, 
the graph of $v (\cdot,t)$ in $B_{2\eta}(d(t)e_n) \cap \overline {\{v >0\}}$ will touch by below the graph of $u$ for a first time $t$, 
and the contact must be an interior point to $B_{2 \eta}(d(t)e_n)$ due to the properties \eqref{e10},\eqref{e13} of $u$ and $v$ (in particular the difference between $ a_n(t) \, (x_n-d(t))^+$ and $v$ is greater than $\eta^{1+\beta}$ on $\p B_{2 \eta}(d(t)e_n)$). 
This contact point is either on the free boundary $\p \{v>0\}$ or on the positivity set $\{v>0\}$ and we reach a contradiction since $v$ is a strict subsolution. The claim \eqref{e12} is proved, hence
$$b(t) \ge \tilde b(t) - C \eta^{1+\beta} \quad \mbox{if} \quad  \quad t \in [-\lambda^{-1}\eta,0].$$
The opposite inequality is obtained similarly and the claim \eqref{e11} holds. Then from \eqref{e10} we deduce that for all $\eta \le c$ small
$$a_n(t) \left(x_n - \tilde b(t) -\eta^{1+\beta'}\right) ^+ \, \, \le \, \, u \, \, \le \, \, a_n(t) \left(x_n -\tilde b(t) + \eta^{1+\beta'} \right)^+ $$
in $B_{\eta} \times [-\lambda^{-1}\eta,0]$
with $\beta'=1/20$ and
$$c \le a_n(t) \le C, \quad |a_n'(t)| \le \eta^{\beta'-2}, \quad \quad \tilde b'(t)=- \lambda a_n(t), \quad \tilde b(0)=0.$$ 
\end{proof}

\section{The Nonlinear problem}

In this section, we use a standard Hodograph transform to reduce our Stefan problem \eqref{SP} to an equivalent nonlinear problem with fixed boundary and oblique derivative boundary condition (see \eqref{NLP}). 

Here and henceforth, for $n \geq 2,$ given $r>0$ we set $$Q_r := (-r,r)^{n}, \quad \quad Q_r^+:=Q_r \cap \{x_n\geq 0\}, \quad Q_r(x_0):=x_0+Q_r,$$
% Let $n \geq 2.$ For $r>0$ set $$Q_r := (-r,r)^{n}, \quad \quad Q_r^+:=Q_r \cap \{x_n>0\}, \quad Q_r(x_0):=x_0+Q_r,$$
$$\mathcal{C}_{r}:= (Q_r \cap \{x_n>0\}) \times (-r, 0], \quad \mathcal{F}_{r}:=\left\{(x,t)| \hspace{0.2cm} x \in Q_r \cap \{x_n=0\},\hspace{0.2cm}t \in (-r,0]\right\}.$$ Also, by parabolic cylinders we mean
$$\mathcal P_r(x_0,t_0):= Q_r(x_0) \times (t_0-r^2,t_0].$$

\subsection{The Hodograph transform} \label{Hod} As mentioned above, we use a Hodograph transform to reduce the Stefan problem \eqref{SP} to one with fixed boundary. Precisely, we view the graph of $u$ in $\R^{n+2}$ 
$$ \Gamma:=\{(x,x_{n+1},t)| \quad x_{n+1}=u(x_1,x_2,\ldots,x_n,t)\}$$ 
as the graph of a possibly multi-valued function $\bar u$ with respect to the $x_n$ direction
$$\Gamma:= \{(x,x_{n+1},t)| \quad x_{n}=\bar u(x_1,x_2,\ldots,x_{n-1},x_{n+1},t)\}.$$ We use $(y_1, \ldots, y_n)$  to denote the coordinates $(x_1,x_2,\ldots,x_{n-1},x_{n+1})$.
Then, if $D u$ and $D \bar u$ denote at some point on the graph $\Gamma$ the gradients with respect to the first $n$ entries of $u$ and $\bar u$, we find
$$D u= - \frac{1}{\bar u_n}  (\bar u_1,\ldots,\bar u_{n-1}, - 1), \quad \quad u_t=-\frac{\bar u_t}{\bar u_n}$$
$$ D^2 u= - \frac{1}{\bar u_n} \left(A(D \bar u)\right)^T \, \, D^2 \bar u \, \, A(D \bar u),$$
where $A(D \bar u)$ is a square matrix which agrees with the identity matrix except on the $n$th row where the entries are given by the right hand side of $D u$ above.  

The Stefan problem \eqref{SP} in terms of $\bar u$ can be written abstractly as the following quasilinear parabolic equation with oblique derivative boundary condition: 
\begin{equation}\label{SPbar}\begin{cases}
\bar u_t = tr (\bar A(\nabla \bar u) \, D^2 \bar u)  &\text{in $\{y_n>0\}$},\\
\bar u_t=g( \nabla \bar u) &\text{on $\{y_n=0\}$},
\end{cases}\end{equation}
with $\bar A(p)$ symmetric, positive definite as long as $p_n \ne 0$, and $g_n(p) >0$. 

The free boundary of $u$ is given by the graph of the trace of $\bar u$ on $\{y_n=0\}$. Our goal becomes to show that $\bar u$ is $C^{1,\alpha}$ with respect to the $y',t$ variables.
Let us assume that $u$ satisfies the hypotheses of Theorem \ref{Main2} (it is now more convenient to work in cubes rather than in balls). 
Below we denote by $c$, $C$ various constants depending on $K$ and $n$. From the flatness assumption
\begin{equation}\label{u-an}
\left|u-a_n(t)(x_n-b(t))^+\right| \le C \eps_1 \lambda \quad \quad \mbox{in} \quad Q_\lambda \times [-\lambda,0],
\end{equation} 
and $0 \in \p \{u>0\}$ implies $|b(0)|\le C\eps_1 \lambda$ which together with $|b'| \le C \lambda$ gives
$$ |b(t)| \le C(\eps_1+ |t|) \lambda.$$
Thus, if $(x,t) \in Q_\lambda \times [-c\lambda, 0]$, then (for $\eps_1$ possibly smaller), $|b(t)| \leq \lambda/2$ and by \eqref{u-an} the domain of definition of $\bar u$ at time $t$ contains $Q^+_{\bar c\lambda}$ for $\bar c$ small enough. We conclude that $\bar u$ is well-defined in $Q^+_{\bar \lambda} \times [-\bar \lambda,0]$, with $\bar \lambda:=c_1 \lambda$, $c_1$ sufficiently small.

Moreover, the graph of $\bar u$ in this set is closed in $
\R^{n+2}$ (since it is obtained as a rigid motion from the graph of $u$) and it satisfies equation \eqref{SPbar} in the viscosity sense, see Definition 
\ref{DefF}  below.

\begin{rem}\label{rem} We observe that $\bar u$ is single-valued in the region $y_n\ge C \eps_1 \lambda,$ and possibly multi-valued near $y_n=0$. Indeed, similarly as above, 
if $t \in [t_0-\lambda^2,t_0+\lambda^2],$ then using the bound for $|b'|$ and \eqref{lala0} for $|a'|$,
$$|a(t)-a(t_0)| \le c_0, \quad \quad |b(t)-b(t_0)| \le C \lambda^2,$$
hence, if $\lambda_0$, $c_0$ are smaller than $\eps_1$ then
\begin{equation}\label{u-an2}
\left|u-a_n(t_0)(x_n-b(t_0))^+\right| \le C \eps_1 \lambda \quad \quad \mbox{in} \quad Q_\lambda \times [t_0-\lambda^2,t_0+\lambda^2],
\end{equation} 
with $ |b(t_0)| \le \lambda/2$.
By applying interior gradient estimates in parabolic cylinders included in $\{u>0\}$ we find from \eqref{u-an2} that if
$$(x_0,t_0) \quad \mbox{with} \quad x_0 \in Q_\lambda, \quad t_0 >-c \lambda \quad \mbox{is in the region} \quad C \eps_1 \lambda \le u(x_0,t_0) \le c \lambda$$
then
$$|\nabla u(x_0,t_0)-a_n(t_0)e_n| \le (2K)^{-1}.$$
\end{rem} Finally, the main hypotheses of Theorem \ref{Main2} can be written in terms of $\bar u$ as
$$|\bar u-(\bar a_n(t)y_n + \bar b(t))| \le C \eps_1 \bar \lambda \quad \mbox{in} \quad Q_{\bar \lambda}^+ \times [-\bar \lambda,0], $$
$$\bar b'(t)=g(\bar a_n(t) e_n), \quad \quad K^{-1} \le \bar a_n \le K,$$
$$\bar \lambda \le \bar \lambda_1, \quad \quad |\bar a'_n| \le \bar c_1 \bar \lambda^{-2}.$$

Our purpose in this paper is to prove an improvement of flatness result for solutions of the nonlinear equation \eqref{SPbar} as above, provided that $\eps_1$, $\bar \lambda_1$, $\bar 
c_1$ are chosen small depending on $n$ and $K$ (see Proposition \ref{P!} in the next section). Then Theorem \ref{Main2} can be obtained by iterating such statement.

\subsection{Assumptions on the nonlinear problem.} We consider solutions to the following problem (for simplicity of notation we drop the bars in our formulation, and we use $x$ rather than $y$),
\begin{equation}\label{NLP}\begin{cases}
u_t = F(\nabla u, D^2u)  &\text{in $\mathcal C_{\lambda},$}\\
u_t=g( \nabla u)& \text{on $\mathcal F_\lambda$.}
\end{cases}\end{equation}
We assume that $F$ is linear in $D^2 u$, that is $F(\nabla u, D^2 u)= tr(A(\nabla u)D^2u)$ and $g_n >0.$

We start by stating precisely the notion of viscosity solution, which can be easily adapted to multi-valued functions $u$ whose graphs are compact sets of $\R^{n+2}$. 

\begin{defn}\label{DefF}We say that a continuous function $u:\overline{\mathcal C}_\lambda \to \R$ is a {\it viscosity subsolution} to \eqref{NLP} if its graph cannot be touched by above at points in $\mathcal C_\lambda \cup \mathcal F_\lambda$ (locally, in parabolic cylinders) by graphs of strict $C^2$ supersolutions $\varphi$ of \eqref{NLP}, i.e.
\begin{equation}\begin{cases}
\varphi_t > F(\nabla \varphi, D^2\varphi)  & \text{in $\mathcal C_{\lambda},$}\\
\varphi_t>g( \nabla \varphi) & \text{on $\mathcal F_\lambda$.}
\end{cases}\end{equation}
\end{defn}
Similarly we can define {\it viscosity supersolutions} and {\it viscosity solutions} to \eqref{NLP}.

%Moreover, these definitions can be easily extended to include possibly multi-valued functions $u$ whose graphs are compact sets of $\R^{n+2}$. 

We define now a class of linear in $x$ functions that we use throughout this paper to express the flatness condition.

\begin{defn}\label{lab}
We denote by $l_{a,b}(x,t)$ functions which for each fixed $t$ are linear in the $x$ variable, and whose coefficients in the $x'$ variable are independent of $t$, and also so that $l_{a,b}$ satisfies the boundary condition in \eqref{NLP} on $\{x_n=0\}$. More precisely,
$$l_{a,b}(x,t):= a(t) \cdot x + b(t), $$
with
$$a(t):= (a_1,\ldots, a_{n-1}, a_n(t)), \quad a_i \in \R,\hspace{0.2cm}i=1,\ldots,n-1,$$
and
$$ b'(t)= g(a(t)).$$
\end{defn}

Our main result is to show that if $u$ is a viscosity solution of \eqref{NLP} which is possibly multi-valued near $\{x_n=0\}$ and is well 
approximated by $l_{a,b}$ in a cylinder $\mathcal C _\lambda$, i.e.
$$|u-l_{a,b}| \le \eps \lambda \quad \mbox{in} \quad \mathcal C_\lambda,$$
then in a smaller cylinder $\mathcal C_{\tau \lambda}$ it can be approximated by 
another function $l_{\tilde a,\tilde b}$ with an error $\eps_\tau=\eps \tau^\alpha$ that improved by a $C^{1,\alpha}$ scaling. 

Before formulating this result rigorously in the next section, we state here the precise hypotheses on $F$ and $g$. We assume that $F(p,M)$ is uniformly elliptic in $M$ for each fixed slope $p\in \R^n$ with $p_n>0$ and 
the ellipticity constants could degenerate as $p_n \to 0^+$ or $|p| \to \infty$. Precisely, for any given constant $K$ large there exists $\Lambda$ large depending on $K$ such that
\begin{equation}\label{AF1}
 \Lambda I \ge D_MF(p,M) \ge \Lambda^{-1} I, \quad \quad \mbox{if} \quad p \in \mathcal R_K,
 \end{equation}
 with
 \begin{equation}\label{rsubk}
 \mathcal R_K:= B_K \cap \{p_n \ge K^{-1}\} \quad \subset \quad \R^{n}.
 \end{equation}
We choose $K$ sufficiently large such that when $p$ is restricted to the set above we also have
\begin{equation}\label{AF2}
|D_p F| \le \Lambda |M|, \quad  \|g\|_{C^1} \le \Lambda, \quad g_n \ge \Lambda^{-1}. 
\end{equation}

From now on we assume that the constants $K$ and $\Lambda$ have been fixed such that \eqref{AF1}-\eqref{AF2} hold. In fact, for notational simplicity, by possibly choosing $K$ larger, we can assume that \eqref{AF1}-\eqref{AF2} hold with $\Lambda=K.$
We consider the situation when $u$ is well approximated in $\mathcal C_\lambda$ by a function $l_{a,b}$ as above with slopes $a(t)$ 
belonging to the region $\mathcal R_K$. 

We suppose in addition that $u$ satisfies the Harnack inequality from scale $\lambda$ to scale $\sigma \lambda$ where $\sigma$ is a 
small parameter. We denote this property for $u$ as property $H(\sigma)$ which is defined in the following way.

\begin{defn}\label{H(d)}
Given a positive constant $\sigma$ small, we say that 
$$\mbox{$u$ has property $H(\sigma)$ in $\mathcal C_\lambda$}$$ 
if $u$ (possibly multi-valued) satisfies the following version of interior Harnack inequality in parabolic cylinders of size $r \in [\sigma \lambda,\lambda]$.

\medskip

Let $l$ denote a linear function
$$l(x):= a \cdot x + b, \quad \quad \mbox{with} \quad a \in \R^n,\hspace{0.2cm}b\in \R, \quad |a|\le K.$$
 If
$$u \ge l \quad  \mbox{ in} \quad  Q_r(x_0) \times [t_0-r^2,t_0+r^2] \quad \quad \subset \quad \mathcal C _\lambda,$$
with $ r \ge \sigma \lambda,$ and $$(u-l)(x_0,t_0) \ge \mu , \quad \mbox{for some} \quad \mu \ge 0,$$ 
then $$u-l \ge \kappa \mu  \quad \mbox{ in} \quad  Q_{r/2}(x_0) \times \left[t_0 +\frac 12 r^2,t_0+r^2\right],$$
for some constant $\kappa$ depending on $n$ and $K$ (but independent of $\sigma$).

Similarly, if $u \le l$ we require these inequalities to hold for $l-u$ instead of $u-l$.

\end{defn}

Property $H(\sigma)$ for all $\sigma>0$ is a consequence of the parabolic Harnack inequality in the case when $u$ is a viscosity solution of \eqref{NLP}, and in addition we know that $\nabla u \in \mathcal R_K$. However, we will show below that property $H(\sigma)$ for some $\sigma$ small, is satisfied for solutions $u$ which are well approximated by functions $l_{a,b}$ and are graphical with respect to the $e_n$ direction.

\section{The iterative statement}

In this section, we state our main improvement of flatness result Proposition \ref{P!}, and we show how Theorem \ref{Main2} can be deduced from it. We also describe the strategy of the proof of Proposition \ref{P!}, and its connection to the corresponding linearized problem \eqref{LP0}.

The improvement of flatness statement reads as follows (we use the notation from Subsection 3.2). The rest of the paper will be devoted to its proof.

\begin{prop}[Improvement of flatness] \label{P!}Fix $K>0$ large, and assume $F$,$g$ satisfy \eqref{AF1}-\eqref{AF2}.  Assume that $u$ is a viscosity solution to \eqref{NLP} possibly multi-valued, which satisfies property $H(\eps^{1/2})$ and
\begin{equation}\label{u-lab}
|u - l_{a,b}| \leq \eps \lambda \quad \text{in $\overline {\mathcal C}_{\lambda}$, with} \quad b'(t)= g(a(t)),
\end{equation}
$$a(t) \in \mathcal R_K, \quad |a'_n(t)| \leq \delta \eps \lambda^{-2},$$ 
and
$$\eps\leq \eps_0, \quad \lambda \leq \lambda_0, \quad  \lambda \le \delta \eps.$$
Then there exists $l_{\tilde a,\tilde b}$ such that
$$|u - l_{\tilde a, \tilde b}| \leq \frac \eps 2 \tau \lambda \quad \text{in $\overline {\mathcal C}_{\tau\lambda}$}, \quad \quad \tilde b'(t)= g(\tilde a(t)),$$
with
$$ |a(t)-\tilde a(t)| \leq C \eps,\quad \quad |\tilde a'_n(t)| \leq \frac{\delta \eps}{2}  (\tau\lambda)^{-2}.$$
Here the constants $\eps_0, \lambda_0, \delta, \tau >0$ small and $C$ large depend only on $n$, and $K$.
\end{prop}

For the remainder of the section constants depending only on $n$ and $K$ are called universal, and denoted by $c_i$, $C_i$.

\begin{rem}
We apply the proposition above to the hodograph transform of a solution to the original Stefan problem, hence in our case $u$ is graphical with respect to the $e_n$ direction. Then \eqref{u-lab} already implies our hypothesis that 
$$ \mbox{$u$ satisfies property $H(\eps^{1/2})$ in $\mathcal C_\lambda$.}$$
Indeed, if $t \in [t_0-\lambda^2,t_0+\lambda^2],$ then using the bounds for $|a'|$, $|b'|$,
$$|a(t)-a(t_0)| \le \delta \eps, \quad \quad |b(t)-b(t_0)| \le C \lambda^2 \le C \delta \eps \lambda,$$
hence
\begin{equation}\label{lat}
|a(t_0)\cdot x + b(t_0) -l_{a,b}| \le C \delta \eps \lambda \quad \mbox{in} \quad Q_\lambda^+ \times [t_0-\lambda^2,t_0+\lambda^2].
\end{equation}
This shows that $u$ is well approximated in each parabolic cylinder of size $\lambda$ by a linear function which is constant in $t$,
\begin{equation}\label{lat2}
|u - (a(t_0)\cdot x + b(t_0))| \leq 2 \eps \lambda \quad \mbox{in} \quad Q_\lambda^+ \times [t_0-\lambda^2,t_0+\lambda^2],
\end{equation}
with $C \ge a_n(t_0) >c$.
Since the graph of $u$ coincides with the graph (in the $e_n$ direction) of a solution to the heat equation, 
we can use the standard Harnack inequality for the heat equation and find that $u$ satisfies property $H(C \eps)$ in $\mathcal C_\lambda$ (as we used interior regularity in Remark \ref{rem}). 
Thus $u$ satisfies property $H(\eps^{1/2})$ by choosing $\eps_0$ smaller if necessary.

This argument shows that if $u$ is graphical with respect to the $e_n$ direction, then it is single-valued away from a $O(\eps \lambda)$ neighborhood of $\{x_n=0\}$.
\end{rem}

We now show that Proposition \ref{P!}  implies Theorem \ref{Main2}, and the remainder of the paper will be devoted to prove Proposition \ref{P!}.

\

{\it Proof of Theorem $\ref{Main2}.$} As discussed in Subsection \ref{Hod}, Theorem \ref{Main2} is equivalent to obtaining $C^{1,\alpha}$ estimates 
on $\{x_n=0\}$ for the hodograph transform. 
After relabeling constants if necessary, the hodograph transform does satisfy the hypotheses of Proposition \ref{P!} 
with 
$\eps=\eps_0$, $\lambda \le \min\{\delta \eps_0,\lambda_0 \}$, $a_0(t)=(0,0,\ldots,0, (a_0)_n(t)) \in \mathcal R_{K/2}$. 
Now Proposition \ref{P!} can be applied indefinitely 
in the cylinders $\mathcal C_{\lambda_k}$, $\lambda_k:=\lambda \tau^k$, with $\eps=\eps_k:=\eps_0 2^{-k}= C(\lambda)\lambda_k^\alpha$. 
The hypothesis that $a_k(t) \in \mathcal R_K$ is satisfied (by choosing $\eps_0$ smaller if necessary) since $$|a_k(t)-a_{k-1}(t)| \le C \eps_k, \quad \quad a_0(t) \in \mathcal R_{K/2},$$ from which we also deduce that 
\begin{equation}\label{akun}|a_k(t) - \nabla u(0,t)| \le C \eps_k. \end{equation}
Hence 
$$|u - l_{a_k,b_k}| \leq \eps_k \lambda_k \le C(\lambda) \, \lambda_k^{1+\alpha} \quad \quad \text{in} \quad  \mathcal C_{\lambda_k},$$
for all $k \ge 0$, and  from \eqref{lat2} (applied for $\lambda_k$) and \eqref{akun} we deduce that 
$$ |\nabla u(0,t) - \nabla u(0,s)| \le C(\lambda)|t-s|^{\alpha /2},$$
which gives
$$|a_k(t)-a_k(s)| \le C(\lambda) \lambda_k^{\alpha/2} \quad \mbox{if} \quad t,s \in [-\lambda_k,0].$$
Using that $b_k'=g(a_k)$ we finally obtain
$$|u-(a_k(0) \cdot x  + b_k'(0)t+ b_k(0))| \le C(\lambda) \lambda_k^{1+ \frac \alpha 2} \quad \quad \text{in} \quad  \mathcal C_{\lambda_k},$$ 
which is the desired conclusion.
\qed

\subsection{Strategy of the proof of the improvement of flatness.} We briefly explain the strategy of the proof of Proposition \ref{P!}. The main idea is to 
linearize the equation near $l_{a,b}$. Define $w (x, t)$ the rescaled error by
\begin{equation}\label{wdef}
 u(x,t):= l_{a,b}(x,t) + \eps \lambda w\left(\frac x \lambda, \frac{t}{\lambda}\right), \quad (x,t)\in \mathcal C_\lambda.
 \end{equation}
Then $w$ is defined in $\mathcal C_1$, possibly multi-valued near $\{x_n=0\}$, and satisfies by hypothesis 
$$|w| \le 1 \quad \mbox{in} \quad \mathcal C_1,$$
and
\begin{equation}\label{nlp}\begin{cases}
\lambda a'_n(\lambda t) x_n + b' (\lambda t) + \eps w_t(x,t) = F\left(a(\lambda t) +\eps \nabla w, \frac \eps \lambda D^2 w\right)&\text{in $\mathcal C_{1},$}\\
\  \\
b'(\lambda t) + \eps w_t = g(a(\lambda t)+ \eps \nabla w)&\text{on $\mathcal F_{1}$.}
\end{cases}\end{equation}

We show that $w$ is well approximated by a solution to the linear equation obtained formally by multiplying the first equation 
by $\lambda \eps^{-1}$ and the second by $\eps^{-1}$ and then letting $\eps \to 0$, $\delta \to 0$. Using $|a'| \le \delta \eps \lambda^{-2}$, and $\lambda \eps^{-1}\le \delta \to 0$ we obtain
\begin{equation}\label{LP0}\begin{cases}
\lambda v_t = tr(A_\lambda(t)D^2v)& \text{in $\mathcal C_{1},$}\\
v_t = \gamma_\lambda(t) \cdot \nabla v&\text{on $\mathcal F_{1}$,}\\
\end{cases}
\end{equation}
with $$A_\lambda(t):=A(a(\lambda t)), \quad \quad \gamma_\lambda(t):=\nabla g(a(\lambda t)).$$
Using that $A,g \in C^2 (\mathcal R_K)$, and that $|a'| \ll \lambda^{-2}$ we find
$$| A'_\lambda(t)| \le \lambda^{-1}, \quad \quad |\gamma'_\lambda(t)| \le \lambda^{-1}. $$
The next sections are devoted to the study of the linear problem \eqref{LP0}, and to obtain estimates which are uniform with respect to $
\lambda$. To this aim, we introduce a distance $d$ between points $(x,t) \in \R^{n+1}$ 
\begin{align*}
d((x,t), (y,s)) : & = \\
=\min\{&|x'-y'|+|x_n-y_n| + |t-s|^{1/2}, \quad |x'-y'|+|x_n|+|y_n|+|t-s|\},
\end{align*}
which is consistent with the scaling of the equation, so that $d$ is equivalent with the standard Euclidean distance on the hyperplane $x_n=0$ and 
with the standard parabolic distance far away from this hyperplane. The various H\"older estimates in the next section are written with respect to this distance $d$, or after a dilation of factor $\lambda^{-1}$ with respect to the rescaled distance $d_\lambda$.
In particular, this allows us to show that solutions $v$ to the linear problem enjoy an improvement of flatness property in cylinders $\mathcal C_{\tau^k}$, which can be 
transferred further to the solutions of the nonlinear problem \eqref{nlp}.

The relation between solutions $w$ to \eqref{nlp} and $v$ to \eqref{LP0} is made precise in the next proposition. It states that $w$ satisfies essentially a comparison principle with $C^2$ subsolutions/supersolutions $v$ of \eqref{LP0} which have bounded derivatives and second derivatives in $x$.

\begin{prop}[Comparison principle]\label{CP}
Let $v \in C^2(\overline{\Omega})$ with $\Omega \subset \mathcal C_1$ satisfy
 $$|\nabla v|, |D^2v| \leq M,$$
 for some large constant $M$ and
\begin{equation}\label{LP01}\begin{cases}
\lambda v_t \leq  tr(A_\lambda(t)D^2v) - C \delta &\text{in $\Omega$,}\\
v_t \leq \gamma_\lambda(t) \cdot \nabla v -  \delta & \text{on $ \mathcal F_{1} \cap \overline{\Omega}$,}\\
\end{cases}
\end{equation}
with $A_\lambda(t)$, $\gamma_\lambda(t)$ as above.

Then
$v$  is a subsolution to \eqref{nlp}, as long as $C$ is sufficiently large, universal, and $\eps \le \eps_1(\delta,M)$. In particular, if 
$$v \leq w \quad \mbox{ on} \quad \overline{\p \Omega \setminus (\{t=0\}\cup \{x_n=0\})}$$ then 
$$v \leq w \quad \mbox{ in} \quad \Omega.$$
\end{prop}

Similarly, we have the same result for supersolutions by replacing $\le$ by $\ge$ and the $-$ signs in \eqref{LP01} by $+$.

\begin{proof}
It is straightforward to show that \eqref{LP01} implies the corresponding inequalities for $v$ (in place of $w$) in \eqref{nlp}. We need to use the hypotheses of Proposition \ref{P!} and that
$$\lambda \|a'\|_{L^\infty} + \|b'\|_{L^\infty} \le C, \quad |A(a(\lambda t)+ \eps \nabla v)-A(a(\lambda t))| \le C \eps M,$$
$$ |g(a(\lambda t)+ \eps \nabla v)-g(a(\lambda t)) - \eps \nabla g(a(\lambda t)) \cdot \nabla v| \le C \eps^2 M^2.$$
\end{proof}

As a consequence, we obtain that if the rescaled error $w$ is close to a $C^2$ solution $v$ of \eqref{LP0} on the {\it Dirichlet boundary} of a domain $\Omega \subset \mathcal C_1$ then $v$ and $w$ remain close to each other in the whole domain $\Omega$.
 
\begin{cor}\label{Com}
Let $w$ be a solution to \eqref{nlp} and $v \in C^2$ be a solution of \eqref{LP0} in a domain $\Omega \subset \mathcal C_1$, with
 $$|\nabla v|, |D^2v| \leq M.$$
If $\eps \le \eps_1(\delta,M)$  and
$$|v - w| \le \sigma \quad \mbox{ on} \quad \overline{\p \Omega \setminus (\{t=0\}\cup \{x_n=0\})}$$ then 
$$|v - w| \le \sigma + C \delta \quad \mbox{ in} \quad \Omega.$$
\end{cor}

\begin{proof} This follows immediately by
applying Proposition \ref{CP} to 
$$ v \pm ( C \delta (x_n^2-t-2)-\sigma). $$
\end{proof}

We apply Proposition \ref{CP} and Corollary \ref{Com} to functions $v$ for which $M$ is large, universal. 
In order to apply Corollary \ref{Com} we need to show that $w$ can be well approximated near the boundary of $\mathcal C_{1/2}$ by a solution $v$ 
to \eqref{LP0} with bounded second derivatives in $x$. We prove that $w$ has essentially a H\"older modulus of continuity (as $\delta \to 0$) 
with respect to the distance $d_\lambda$ induced by $d$, and then we let $v$ be the solution to the Dirichlet problem \eqref{LP0} in $\mathcal 
C_{1/2}$ with boundary data which is sufficiently close to $w$.

\smallskip

We conclude this section by stating a version of interior Harnack inequality for $w$ with respect to constants, which is an immediate consequence of property $H(\eps^{1/2})$ of $u$ in $\mathcal C_\lambda$, see Definition \ref{H(d)}. 

As in \eqref{lat}, the error between $l_{a,b}$ and a linear function independent of $t$ in a time-interval of size $(\lambda r)^2$ is $C \delta \eps \lambda \, r^2$. Then Definition \ref{H(d)} implies the following property for $u-l_{a,b}$.

\medskip

If for some constant $\omega$
$$u -(\omega+l_{a,b}) \ge 0 \quad  \mbox{ in} \quad  Q_{\lambda r}(x_0) \times [t_0- (\lambda r)^2,t_0+ (\lambda r)^2] \subset \mathcal C _\lambda,$$
with $ r \in [\eps^{1/2},1]$, and $$(u-(\omega+l_{a,b}))(x_0,t_0) \ge \mu \eps \lambda , \quad \mbox{for some} \quad \mu \ge C \delta r^2,$$ 
then $$u-(\omega+l_{a,b}) \ge \frac{\kappa}{2} \mu \eps \lambda \quad \mbox{ in} \quad  Q_{r\lambda/2}(x_0) \times \left[t_0 +\frac 12 (\lambda r)^2,t_0+ (\lambda r)^2\right],$$
with $\kappa$ the universal constant from Definition \ref{H(d)}. In terms of $w$ this can be written as follows.

\

{\it Interior Harnack inequality for $w$.} If 
$$ w \ge \omega \quad  \mbox{ in} \quad  Q_r(x_0) \times [t_0- \lambda r^2,t_0+ \lambda r^2]  \subset  \mathcal C _1,$$
with $\omega$ a constant, $ r \ge \eps^{1/2},$ and $$w(x_0,t_0) \ge \omega + \mu, \quad \mbox{for some} \quad \mu \ge C \delta r^2,$$ 
then 
\begin{equation}\label{ihiw}
w \ge \omega + \frac{\kappa}{2} \mu \quad \mbox{ in} \quad  Q_{r/2}(x_0) \times \left[t_0 +\frac \lambda 2 r^2,t_0+\lambda r^2\right].
\end{equation}

\comment{By our assumption on the coefficients and the linearity of $F $ in $D^2 w,$ we deduce that
\begin{equation}\label{LPP}\begin{cases}
w_t = F(a(\lambda^2 t) +\eps \nabla w,  D^2 w) \pm O(\delta)&\text{in $C_{1, \frac 1 \lambda}$}\\
\  \\
w_t = \lambda \nabla g(a(\lambda^2 t) \cdot \nabla w+ \lambda o(\eps |\nabla w|^2),& \text{on $F_{1,1/\lambda}$.}
\end{cases}\end{equation}}

\section{The linearized problem}

In this section, we state various estimates for the linear problem \eqref{LP0} which are uniform in the parameter $\lambda \le 1$ and we use them to prove our main result Proposition \ref{P!}. We start with introducing the distance $d_\lambda$ with respect to which our estimates are obtained.

\subsection{ Definition of the distances $d$, $d_\lambda$ and the family of balls $\mathcal B_r$, $\mathcal B_{\lambda,r}$.} 

We define the following distance in $\R^{n+1}$
\begin{align*}
d((x,t), (y,s)) : & = \\
=\min\{&|x'-y'|+|x_n-y_n| + |t-s|^{1/2}, \quad |x'-y'|+|x_n|+|y_n|+|t-s|\},
\end{align*}
which interpolates between the parabolic distance and the standard one depending on how far points are from $\{x_n=0\}$. It is not too difficult to check that $d$ satisfies the triangle inequality.

For $r \le 1$ and points $(y,s)$ with $y_n \in [0,1],$ we define the family of ``balls" of center $(y,s)$ and radius $r,$ which are backwards in time and restricted to $\{x_n \ge 0\},$ and which are consistent with the distance induced by $d$: 
\begin{align*}
&\mathcal B_r(y,s):= Q_r(y) \times (s-r^2,s),& &\mbox{if $r<|y_n|$},\\
&\mathcal B_r(y,s):= Q^+_r(y) \times (s-r,s),& &\mbox{if $1 \ge r \ge |y_n|$,}
\end{align*}
where we recall that $$Q_r(y):=\{x \in \R^n| \, |x_i-y_i|<r \}, \quad \quad Q^+_r(y):=Q_r(y) \cap \{ x_n \ge 0\}.$$
Notice that $$(x,t) \in \mathcal B_{2r}(y,s) \setminus \mathcal B_r(y,s) \quad \Longrightarrow \quad d((x,t),(y,s)) \sim r.$$ 

A function $v: \overline U \to \R$, with $U\subset \mathcal C_1,$ is H\"older with respect to the distance $d$ if
$$[v]_{C^\alpha_d}:=\sup_{(x,t) \ne (y,s)} \,  |v(x,t)-v(y,s)| \, \, d((x,t),(y,s))^{-\alpha} < \infty.$$
Equivalently, $v \in C^\alpha_d(\overline U)$ if and only if there exists $M$ such that $\forall (x,t) \in \overline U$
$$ osc \, \, \, v \le M r^\alpha \quad \mbox{in} \quad \mathcal B_r(x,t) \cap \overline U.$$

\

{\bf Rescaling.} Assume $\lambda \le 1$ and we perform a dilation of factor $\lambda^{-1}$ which maps $Q_\lambda^+$ into $Q_1^+$. We use hyperbolic scaling for the rescaled distance $d_\lambda$ of $d$ 
$$d_\lambda ((x,t), (y,s)) : = \frac 1 \lambda d(\lambda(x,t), \lambda (y,s))$$
$$=\min\{|x'-y'|+|x_n-y_n| + \lambda ^{-1/2}|t-s|^{1/2}, |x'-y'|+|x_n|+|y_n|+|t-s|\}.$$
The corresponding family of balls induced by $d_\lambda$ denoted by $\mathcal B_{\lambda,r}$ is obtained by dilating of a factor $\lambda^{-1}$ the sizes of the balls $\mathcal B_r$ above and then relabeling $\lambda^{-1}r$ by $r$. We find
\begin{align*}
&\mathcal B_{\lambda,r}(y,s):= Q_r(y) \times (s- \lambda r^2,s),& &\mbox{if $r<|y_n|$},\\
&\mathcal B_{\lambda,r}(y,s):= Q^+_r(y) \times (s-r,s),& &\mbox{if $\lambda^{-1} \ge r \ge |y_n|$,}
\end{align*}
and notice that $\mathcal B_{\lambda,r}(y,s)=\mathcal B_{r}(y,s)$ if $y_n=0$.

As above a function $v$ is H\"older with respect to the distance $d_\lambda$ in $\overline U$ and write $v \in C_{d_\lambda}^\alpha(\overline U)$ if there exists $M$ such that
$$ osc \, \, \, v \le M r^\alpha \quad \mbox{in} \quad \mathcal B_{\lambda,r}(x,t) \cap \overline U.$$

\subsection{Estimates.} Having introduced the distance $d_\lambda$, we are now ready to state the estimates for the linear problem
\begin{equation}\label{LP}\begin{cases}
\lambda v_t = tr(A(t) D^2v) &\text{in $\mathcal C_{1}$,}\\
v_t = \gamma(t) \cdot \nabla v &\text{on $\mathcal F_{1}$,}
\end{cases}
\end{equation}
with
$$K^{-1}I \le A(t) \le KI, \quad \quad K^{-1} \leq \gamma_n \leq K, \quad \quad |\gamma| \leq K$$
$$\lambda \in(0, 1], \quad \quad  |A'(t)| \le  \lambda^{-1}, \quad \quad |\gamma'(t)| \le \lambda^{-1},$$
for some large constant $K$. Here constants depending on $n$ and $K$ are called universal.

We start with an interior regularity result  (see Definition $\ref{lab}$ of $l_{a,b}$).

\begin{prop}[Interior estimates] \label{estimate} Let $v$ be a viscosity solution to \eqref{LP} such that $\|v\|_{L^\infty} \leq 1.$ Then 
$$| \nabla v|, \, \, \,  |D^2 v| \, \, \le C \quad \quad \mbox{in} \quad \mathcal C_{1/2},$$
and for each $\rho \le 1/2$, there exists $l_{\bar a,\bar b}$ such that 
$$|v - l_{\bar a, \bar b}|\leq C \rho^{1+\alpha} \quad \quad \text{in $\mathcal C_{\rho}$},$$
with 
$$\bar b'(t)= \gamma(t) \cdot \bar a, \quad \quad |\bar a'_n| \leq C \rho^{\alpha-1}\lambda^{-1}, \quad \quad |\bar a| \le C,$$ with $\alpha$, $C$ universal.

\end{prop}

In terms of the Dirichlet problem for \eqref{LP}, we define the {\it Dirichlet boundary} of $\mathcal C_1$ as 
$$\p_D \mathcal C_1:= \p \mathcal C_1 \cap \left( \{t=-1\}  \cup \{x_n=1\} \cup_{i=1}^{n-1} \{|x_i|=1 \}\right).$$ 

Notice that $\p_D \mathcal C_1$ is different from the standard parabolic boundary since the points on $\mathcal F_1$ are also excluded. 

\begin{prop}[The Dirichlet problem] \label{DiPr} Let $\phi$ be a continuous function on $\p_D \mathcal C_1$. Then there exists a unique classical solution $v \in C^{2,1} (\mathcal C_1) \cap C^0(\bar {\mathcal C_1})$ to the Dirichlet problem \eqref{LP} with $v=\phi$ on $ \p_D \mathcal C_1$. Moreover,
$$|\nabla v|, |D^2 v| \le C(\sigma) \|v\|_{L^\infty} \quad \mbox{in} \quad C_1^\sigma:=\{ d_\lambda((x,t), \p_D \mathcal C_1) \ge \sigma\},$$
and if $\phi$ is $C^\alpha$ with respect to the distance $d_\lambda,$ then $v$ is also $C^\alpha$ up to the boundary and 
$$\|v\|_{C^
\alpha_{d_\lambda}} \le C \|\phi\|_{C^\alpha_{d_\lambda}},$$
with $C(\sigma)$, $C$ universal constants (independent of $\lambda$).
\end{prop}

Here $$\|v\|_{C^\alpha_{d_\lambda}} := \|v\|_{L^\infty} +\sup_{(x,t) \ne (y,s)} \,  |v(x,t)-v(y,s)| d_\lambda((x,t),(y,s))^{-\alpha}.$$

The proofs of Propositions \ref{estimate} and \ref{DiPr} are based on a Harnack inequality for solutions to \eqref{LP}, which we provide in the next section. The Harnack inequality holds for more general 
equations of the same type with measurable coefficients. It applies also for solutions $w$ to the nonlinear problem \eqref{nlp} up to scale $\eps^{1/2}$.  
To state it, we recall the definition of the maximal Pucci operators
\begin{equation}\label{Pucci} \mathcal M_K^+ (N) = \max_{K^{-1} I \le A \le K I} \quad tr \, \,  A N, \quad \quad  \quad \mathcal M_K^- (N) = \min_{K^{-1} I \le A \le K I} \quad tr \, \, A N.\end{equation}

\begin{thm}[H\"older continuity]\label{H} Let $v$ be a viscosity solution to 
\begin{equation}\label{LP*}\begin{cases}
\mathcal M^+_K(D^2 v) \ge \lambda v_t \ge \mathcal M^-_K(D^2 v) &\text{in $\mathcal C_{1}$},\\
\ \\

K^{-1} v_n^--K v_n^+ - K |\nabla_{x'}v| \ge v_t \ge  K^{-1} v_n^+-K v_n^- - K |\nabla_{x'}v|  &\text{on $\mathcal F_{1}$.}
\end{cases}
\end{equation}
Then $v$ is locally H\"older continuous in $\mathcal C_{1/2}$ with respect to the metric induced by $d_\lambda,$ that is
$$\|v\|_{C_{d_\lambda}^\alpha(\mathcal C_{1/2})} \leq C \|v\|_{L^\infty(\mathcal C_1)}.$$
Moreover, if $v$ is continuous up to the boundary and $v=\phi$ on $ \p_D \mathcal C_1$ with $\phi \in C_{d_\lambda}^\alpha$ then $v \in C^\alpha_{d_\lambda}$ up to the boundary and
$$\|v\|_{C^
\alpha_{d_\lambda}} \le C \|\phi\|_{C^\alpha_{d_\lambda}}.$$
The constants $\alpha$ and $C$ depend only on $n$ and $K$.
\end{thm}

\begin{prop}[Harnack inequality for $w$]\label{Hanw} 
Assume that $u$ satisfies the hypotheses of Proposition $\ref{P!}$ and $w$ is defined as in \eqref{wdef}. Then 
$$osc _{\, \, \mathcal B_{\lambda,r}(x_0,t_0)} \, w \le C r^\alpha, \quad \quad \quad \quad \forall (x_0,t_0) \in \mathcal C_{1/2},\quad r \ge C(\delta) \eps^{1/2},$$
provided that $\delta\le c'$ universal.
\end{prop}

\subsection{Proof of Proposition $\ref{P!}$.} Using the results above we can complete the proof of Proposition \ref{P!}.

\begin{proof}[Proof of Proposition $\ref{P!}$.] We divide the proof in two steps.

\medskip

{\bf Step 1.} We prove that there exists a solution $v$ to \eqref{LP0} which approximates $w$ well in $\mathcal C_{1/2}$, that is
$$|v-w| \leq C \delta  \quad \text{in $\mathcal C_{1/2}$},$$
provided that $\eps \le \eps_1(\delta)$.

 Indeed, by Proposition \ref{Hanw} we know that there exists a function $\phi$ defined in $\mathcal C_{1/2}$ such that 
\begin{equation}\label{wdelta}|w-\phi| \leq \delta, \quad \quad \|\phi\|_{C^\alpha_{d_\lambda}} \leq C.\end{equation}
Let $v$ be the solution to \eqref{LP0} in $ \mathcal C_{1/2}$ with $v=\phi$ on $\p_D \mathcal C_{1/2}$, which exists in view of Proposition \ref{DiPr} and satisfies,
\begin{equation}\label{vdelta}\|v\|_{C^\alpha_{d_\lambda}} \leq C.\end{equation}
Then, if $d_\lambda((x,t), \p_D \mathcal{C}_{1/2}) \leq \delta^{1/\alpha}$, there exists $(y,s)$ on  $\p_D \mathcal{C}_{1/2}$ so that (using \eqref{vdelta} and \eqref{wdelta}),
$$|v(x,t) - \phi(y,s)| \leq C\delta, \quad \quad |w(x,t) - \phi(y,s)|\leq  C\delta,$$
thus,
\begin{equation}\label{sigma}|v-w| \leq C \delta \quad \quad \text{on} \quad   \mathcal C_{1/2} \cap \{d_\lambda((x,t), \p_D \mathcal C_{1/2}) \leq \delta^{1/\alpha}\}.\end{equation} In particular 
$$|v-w| \leq C \delta \quad \text{on} \quad \p_D \Omega, \quad \quad \Omega:=\mathcal C_{1/2} \cap \{d_\lambda((x,t), \p_D \mathcal{C}_{1/2}) > \delta^{1/\alpha}\}.$$
On the other hand, by Proposition \ref{DiPr}, $$|\nabla v|, |D^2 v| \leq C(\delta) \quad \quad \mbox{in} \quad \Omega.$$
Thus, using Corollary \ref{Com}, 
$$|v-w| \leq C \delta \quad \text{in} \quad \Omega,$$ which gives the desired claim.

\

{\bf Step 2.} Applying Proposition \ref{estimate}, to the solution $v$ above, we find that
$$|w- l_{\bar a, \bar b}|\leq  C \rho^{1+\alpha} + C \delta \quad \text{in} \quad \mathcal C_{\rho},$$
and
$$\bar b'(t)= \gamma_\lambda(t) \cdot \bar a, \quad |\bar a'_n| \leq C \rho^{\alpha-1} \lambda^{-1}, \quad \quad |\bar a | \leq C,$$
with $\gamma_\lambda(t)=\nabla g(a( \lambda t))$.
We choose $\rho=\tau$ small, universal, and $$\delta=\tau^{1+ \frac \alpha 2},$$ so that $\delta\le c'$ the constant from Proposition \ref{Hanw}, and
$$|w - l_{\bar a, \bar b}|\leq \frac 1 4 \tau \quad \text{in $\mathcal C_{\tau}$},\quad \quad  |\bar a'_n| \leq \frac 1 4 \, \delta \, \, \tau^{-2} \lambda^{-1}.$$
In terms of the original function $u$, this inequality implies
$$\left|u - \left(l_{a,b} + \eps \lambda l_{\bar a, \bar b}\left(\frac x \lambda, \frac{t}{\lambda}\right)\right)\right| = \eps \lambda \left|w\left(\frac x \lambda, \frac{t}{\lambda}\right) -  l_{\bar a, \bar b}\left(\frac x \lambda, \frac{t}{\lambda}\right)\right| \leq \frac{\eps}{4} \tau \lambda \quad \text{in} \quad  \mathcal C_{ \tau \lambda}.$$
Set
$$\tilde a (t) : = a(t) + \eps \, \bar a \left(\frac{t}{\lambda}\right), \quad \hat{b}(t):= b(t) + \eps \lambda \, \bar b\left(\frac{t}{\lambda}\right),$$
then
$$|u-l_{\tilde a, \hat b}|  \leq \frac{\eps}{4} \tau \lambda \quad \text{in} \quad  \mathcal C_{ \tau \lambda},$$
and 
$$|\tilde a_n'| \leq \frac {\eps \delta}{\lambda^2}\left(1  + \frac{1}{4\tau^2}\right) \leq  \frac {\eps \delta}{2(\tau\lambda)^2}.$$
Finally, we define $\tilde b$ by the ODE
$$\tilde b'=g(\tilde a), \quad \quad \tilde b(0)=\hat b (0),$$
and then we have
$${\hat b}'= b' + \eps  \bar b'\left(\frac{t}{\lambda}\right)= 
g(a(t)) + \eps \nabla g(a(t)) \cdot \bar a \left(\frac{t}{\lambda}\right)=g(\tilde a(t))+O(\eps^2)=\tilde b' + O(\eps^2).$$
If $t \in [-\tau \lambda,0]$ then
$$|(\tilde b - \hat b)(t)| \le C \eps^2 |t| \le \frac{\eps}{4} \tau \lambda,$$
which implies the desired conclusion
$$|u-l_{\tilde a, \tilde b}| \leq \frac \eps 2 (\lambda \tau) \quad \text{in} \quad  \mathcal C_{ \tau \lambda},$$
and $\tilde a, \tilde b$ satisfy the required bounds.
\end{proof}

\section{Harnack inequality}

In this section, we prove Theorem \ref{H} and Proposition \ref{Hanw}. The key ingredient is to establish a diminishing of oscillation property. As usual, universal constants depend on $n,K.$

\begin{prop}\label{dim_osc} Assume that $v$ is a viscosity solution of \eqref{LP*} and $0 \leq v \leq 1$ in $\mathcal C_{1}$. Then
$$osc_{\mathcal{C}_{1/2}} v \leq 1-c,$$
with $c>0$ universal.
\end{prop}

In order to prove Proposition \ref{dim_osc} we start with a lemma. Let $\Omega$ be a smooth domain in $\R^n$, $n \geq 2,$ such that
$$\bar Q^+_{3/4} \subset \bar \Omega \subset \bar Q^+_{7/8},$$
and call
$$T:= \{x_n=0\} \cap Q_{3/4} \subset \p\Omega.$$ 
Define $\eta(x')$ a standard bump function supported on $Q'_{5/8}$ and equal 1 on $Q'_{1/2}$ (here the prime denotes cubes in $\R^{n-1}$).
%Finally, denote,
%$$D:= B_{1/4}(\frac{e_n}{2}).$$
Let $\phi$ satisfy (see \eqref{Pucci} for the definition of the Pucci operator),
$$\mathcal M^-_{K}(D^2\phi)= 0 \quad \text{in $\Omega$},$$
$$\phi= 0 \quad \text{on $\p \Omega \setminus T$}, \quad \phi=\eta \quad \text{on $T$},$$ and notice that $0 \le \phi \le 1$, $\phi \geq c$ on $Q_{1/2}^+$, and by Hopf lemma $\phi_n>0$ on $\{x_n=0\} \cap \{\phi=0\}$. The following lemma holds.

\begin{lem}\label{phi} Let $v \geq 0$ satisfy
\begin{equation}\label{LP1}\begin{cases}
\mathcal M^+_K(D^2 v) \ge \lambda v_t \ge \mathcal M^-_K(D^2 v)&\text{in $\mathcal C_{1}$,}\\
v_t \ge  K^{-1} v_n^+-K v_n^- - K |\nabla_{x'}v| &\text{on $\mathcal F_{1}$,}
\end{cases}
\end{equation}
in the viscosity sense. If for some $t_0 \in (-1,0]$,
$$v(x, t_0) \geq s_0 \, \, \phi(x) \quad \text{in $Q^+_1$}, \quad s_0 \geq 0,$$
then
$$v(x,t) \geq s(t) \, \phi(x) \quad \text{in $Q^+_1 \times [t_0, 0],$}$$
with
$$s'(t)= - C_0 s(t), \quad s(t_0)=s_0, \quad \quad \mbox{$C_0$ large universal}.$$ 
Moreover, if $s_0 \le c_0$ with $c_0$ small universal, and
\begin{equation}\label{center}v\left(\frac 1 2 e_n, t_0+\lambda/4 \right) \geq \frac 12,\end{equation}
then 
$$v(x,t_0+\lambda) \geq (s_0 + c_0\lambda) \phi(x).$$
\end{lem}

\begin{proof} For the first part of the claim, since $v \geq 0$, it suffices to show that with our choice of $s$,
$$w(x,t):=s(t)\phi(x),$$ is a subsolution to \eqref{LP1} in $\Omega \times [t_0,0]$, that is
$$
\begin{cases}
\lambda w_t \le \mathcal M^-_K(D^2 w) &\text{in $\Omega \times (t_0,0]$,}\\
w_t \le  K^{-1} w_n^+-K w_n^- - K |\nabla_{x'}w|  &\text{on $\{x_n=0\} \cap (\Omega \times (t_0,0])$.}
\end{cases}
$$
The interior equation is immediately satisfied since $s' \leq 0$ and $s \geq 0.$ On $\{x_n=0\}$, we need to show that
$$  C \phi + K^{-1} \phi_n^+ -K \phi_n^- - K |\nabla_{x'}\phi| \ge 0,$$
for some large $C$.
By Hopf lemma $\phi_n >0$ on $\{\phi=0  \}\cap \{x_n=0\}$ and moreover $|\nabla_{x'} \phi|=0$, thus
$$K^{-1} \phi_n^+- K \phi_n^-- K |\nabla_{x'}\phi| = K^{-1} \phi_n > 0 \quad \mbox{on} \quad \{\phi=0  \}\cap \{x_n=0\}.$$ The same holds in a neighborhood of this set by continuity, and then we can choose $C$ sufficiently large so that the desired inequality holds.

% $\phi_n \geq c$ universal on $T \cap B^c_{5/8-\delta}$, $\delta$ small enough so that,
%$$s' \phi \leq 0 \leq s \gamma \cdot \nabla \phi.$$ This inequity is also trivially satisfied on $T \cap B_{1/2 -\delta}.$ On the remaining part of the boundary, it suffices to ask that $$s' \leq  -C s,$$ for some $C>0$ universal.

For the second part, denote for simplicity
$$t_i:=t_0 + i \frac{\lambda}{4}, \quad i=1,\ldots,4.$$
 We define $$D:= \{x \in \Omega| \quad d(x, \p \Omega) > c\} \subset \Omega,$$
 with $c$ small universal such that there exists a $C^2$ function $\psi \ge 0$ defined in $\Omega \setminus D$ satisfying
 $${\mathcal M}^-_K (D^2\psi) \ge 4 \quad \mbox{in} \quad \Omega \setminus D,$$
 and
 $$ \psi=0, \quad |\nabla \psi|\ge 1 \quad \mbox{on $\p \Omega$,} \quad \psi \le 1 \quad \mbox{on $\p D$.}  $$
 An example of such a function is given by $\psi=d+ C d^2$ with $C$ sufficiently large, where $d$ is the distance function to $\p \Omega$.
 In view of \eqref{center}
 $$v\left(\frac 1 2 e_n, t_1\right) \geq 1/2.$$ Thus,
we can use Harnack inequality (after rescaling) to conclude that
\begin{equation}\label{v2c1}
v \geq 2c_1 \quad \text{on $D \times [t_2,t_4]$},
\end{equation}
for some small $c_1$.
We claim that at time $t=t_3$,
\begin{equation}\label{vxt3}
v(x,t_3) \ge s(t_3)\phi + c_1 \psi \quad \mbox{in} \quad \Omega \setminus D.
\end{equation}
 For this we compare $v$ in $(\Omega \setminus D) \times [t_2,t_3]$ with
$$q(x,t):=s(t_3)\phi + c_1\left (\psi + \frac{t-t_3}{t_3-t_2} \right )  .$$ The inequality $q \le v$ holds on the boundary of the domain. Indeed (recall that $s$ is decreasing), on $\p D$ 
$$q(x,t) \le s(t_3) \phi + c_1 \le s_0 + c_1 \le 2c_1\leq v,$$ 
where in the last inequality we used \eqref{v2c1},
and on $\p \Omega$ or at $t = t_2$ we have $q \le s(t_3) \phi \le v$.

It remains to check that $q$ is a subsolution for the interior equation. Indeed,
$$\lambda q_t = 4 c_1 \le c_1\, \mathcal M^-_K  (D^2 \psi) \le \mathcal M^-_K (D^2 q),$$
where we used that $ \mathcal M_K^-(N_1) + \mathcal M_K^-(N_2) \le \mathcal M_K^-(N_1+N_2)$, and claim \eqref{vxt3} is proved.

Next, in the domain $(\Omega \setminus D) \times [t_3,t_4]$ we compare $v$ with the subsolution 
$$z(x,t):= (s(t_3)+ c_2 (t-t_3))\phi(x) + c_1\psi(x),$$ with $c_2$ sufficiently small. 

The inequality $v \ge z$ is satisfied at time $t=t_3$ by \eqref{vxt3}, and on $\p D$ we have
$$z \le s_0 + c_2 + c_1 \le 2 c_1 \le v,$$
while on $\p \Omega \setminus \{x_n=0\}$ we have $z=0\le v$. 
We check that $z$ is a subsolution of our problem. For the interior inequality we have
$$\lambda z_t =c_2 \lambda \phi \le c_2 \le c_1\, \mathcal M^-_K  (D^2 \psi) \le \mathcal M^-_K  (D^2 z).$$
For the boundary condition, on $\{x_n=0\}$ we get
\begin{equation}\label{zt} z_t=c_2 \phi \le c_2 \le \frac {c_1}{4} K^{-1}  \psi_n,\end{equation}
where in the second inequality we have used that $\psi_n \ge 1$ on $\p \Omega \cap \{x_n=0\}$. Moreover, since $\phi_n \geq - C$ on $\p \Omega \cap \{x_n=0\}$, we get (for $s_0,c_2$ small enough),
$$z_n \ge -\left(s_0 + c_2 \frac{\lambda}{4}\right) C + c_1 \psi_n \geq \frac{c_1}{2} \psi_n ,$$ and finally ($|\nabla_{x'} \psi|=0$ on $\{x_n=0\}$)$$K |\nabla_{x'}z| \le \left(s_0+\frac{c_2}{4}\right) K |\nabla_{x'}\phi| \le \frac {c_1}{4} K^{-1}  \psi_n .$$
Together with \eqref{zt}, this gives $$ z_t=c_2 \phi \le c_2 \le K^{-1} z_n - K |\nabla_{x'}z| \quad \text{on $\{x_n=0\}.$}$$
In conclusion, at time $t=t_4$ we have $v\ge z$ in $\Omega \setminus D$ and $v \ge 2 c_1$ in $D$ which gives the desired claim by choosing $c_0$ sufficiently small.
\end{proof}

\begin{rem}\label{rem01}
In the proof above we only used the subsolution property for $v$ 
\begin{equation}\label{m+L}
\mathcal M^+_K (D^2v) \ge \lambda v_t,
\end{equation}
in order to extend the inequality \eqref{center} from one point to \eqref{v2c1} by applying the interior parabolic Harnack inequality. Alternately, it is sufficient to assume that the Harnack inequality  holds for $v$ only in a neighborhood of $D$ and not necessarily up to $\{x_n =0\}$. 

The rest of the proof is based on comparing $v$ with the explicit $C^2$ subsolutions $w$, $q$ and $z$ which all have bounded second derivatives in 
the $x$ variable. Thus the hypothesis that $v$ is a viscosity supersolution of \eqref{LP1} can be slightly relaxed, and require instead, that $v$ only 
satisfies the 
comparison principle with respect to the explicit barriers above.
\end{rem}

\begin{rem}\label{rem02}

The hypothesis \eqref{m+L} can be removed completely if instead of \eqref{center} we assume a measure estimate
$$\left|\left\{v \ge \frac 14\right\} \cap \left(Q_1 \times \left[t_0,t_0+ \frac \lambda 4\right]\right)\right| \ge \frac 12 \left|Q_1 \times \left[t_0,t_0+ \frac \lambda 4\right] \right|.$$
Then, the inequality \eqref{v2c1} follows directly from the supersolution property for $v$ and the weak Harnack inequality (see for example \cite{W}). 
\end{rem}

We are now ready to prove Proposition \ref{dim_osc}.

\

{\it Proof of Proposition $\ref{dim_osc}$.} Assume that $0 \le v \le 1$, and for half of the values of 
$$ t_k:= -1+ 
k\lambda, \quad \quad \mbox{so that} \quad t_k \in [-1, -1/2), \quad k=0,1,2,\dots,$$
we have
\begin{equation}\label{vfra}
v\left(\frac 12 e_n,t_k+ \lambda/4\right) \ge \frac 12.
\end{equation}
We apply Lemma \ref{phi} repeatedly to the sequence of times $t_k$ and obtain
$$v(x,t_k) \geq s_k \phi, \quad s_k:=s(t_k), \quad \quad s_0=0,$$
with $\phi$ given in Lemma \ref{phi}, and
$$s_{k+1} \ge s_k + c_0 \lambda \quad \quad \mbox{if \eqref{vfra} holds and $s_k \le c_0,$} $$
or
$$s_{k+1} \ge s_k (1-C_0 \lambda) \quad \mbox{otherwise.}$$
 Now it follows that $s_k \ge c_1$ for the last value of $k$ so that $t_k < - 1/2$, 
for $c_1$ appropriately  chosen depending on $c_0, C_0.$
Then we apply the first part of Lemma \ref{phi} to obtain
$$v(x,t) \geq \bar c \phi \quad \text{for all $t\geq -1/2$},$$
which gives the desired conclusion, since $\phi > c$ on $Q_{1/2}^+.$
\qed

\

The same arguments show that a similar statement to that of Proposition \ref{dim_osc} holds for a solution $w$ of \eqref{nlp} defined in \eqref{wdef}. Below is the key lemma which connects the linear and nonlinear problem and allows us to reduce our analysis mostly to the linear case.

\begin{lem}\label{dim_osc_w}
Assume that $u$ satisfies the hypotheses of Proposition $\ref{P!}$ and let $w$ be defined as in \eqref{wdef}, with $-1 \leq w \leq 1$. Then 
$$ osc_{\, \, \mathcal C_{1/2}}w \le 2(1-c),$$
with $c$ universal, provided that $\delta \le c'$ and $\eps \le \eps_1(\delta)$.
\end{lem}

\begin{proof}
We may assume as above that $w(e_n/2,t_k+ \lambda/4)\ge 0$ for more than half the values of $k$, and then show that $w$ separates from the lower 
constraint $-1$. For this we apply the same argument as above 
to the function
$$ \bar w:=w + 1 + C \delta (2 + t -x_n^2) \ge 0,$$
for which the relaxed hypotheses of Remark \ref{rem01} hold. Indeed, by \eqref{ihiw}, $\bar w$ satisfies the required Harnack inequality \eqref{center} $\Longrightarrow$ \eqref{v2c1} and, by Proposition \ref{CP}, it satisfies the comparison with the explicit barriers of Lemma \ref{phi}.  

We remark that we have only used that $u$ has property $H(c'')$ in $\mathcal C_\lambda$ for some $c''$ small, universal.
\end{proof}

\smallskip

Before we proceed with the proofs of Theorem \ref{H} and Proposition \ref{Hanw} we provide a boundary version of the diminishing of oscillation Proposition \ref{dim_osc}.

\begin{lem}\label{phiB}
Assume that $U$ is a space-time domain obtained by the intersection of $n+1$ half spaces in the $x_1$, $\ldots,$ $x_{n-1}$, $x_n$ and $t$ variables, 
$$U:=(-\infty,z_1) \times (-\infty, z_2) \times \dots \times  (-\infty, z_n) \times (-z_{n+1},\infty) \quad \subset \R^{n+1},$$ 
with $z_i \in [0,1]$.

Assume that $v \geq 0$ satisfies
\begin{equation}\label{LP2}\begin{cases}
\lambda v_t \ge \mathcal M^-_K(D^2 v)&\text{in $\mathcal C_{1}\cap U$,}\\
v_t \ge  K^{-1} v_n^+-K v_n^- - K |\nabla_{x'}v|& \text{on $\mathcal F_{1} \cap U$,}\\
v \ge \frac 14&\mbox{on $\p U \cap \mathcal C_1.$}
\end{cases}
\end{equation}
If $\min z_i \le \frac 7 8,$ then 
$$ v \ge c \quad \mbox{in $ \mathcal C_{1/2} \cap U$,} \quad \quad \mbox{$c$ universal}.$$
\end{lem}

\begin{proof}
This follows easily from Lemma \ref{phi}. Indeed, we work with the truncation $\tilde v:=\min \{v,\frac 14\}$ extended by $\frac 14$ in $\mathcal C_1 \setminus U$. Then $\tilde v$ is a supersolution for our problem in $\mathcal C_1$.

If $z_{n+1} < 1$, then we can apply directly the first part of Lemma \ref{phi} for $\tilde v$ for some $t_0$ close to $-1$ and for $s_0$ universal, and obtain the desired conclusion. 

On the other hand, if $z_{n+1}=1$, then $z_i \le \frac 78$ for some $i \le n$ hence for each time $t \in [-1,0]$ we find
$$\left|\left\{\tilde v \ge \frac 14 \right\} \cap Q_1\right| \ge c |Q_1|.$$
 Now the conclusion follows as before, see Remark \ref{rem01}.
\end{proof}

We are now ready to prove Theorem \ref{H}.

\begin{proof}[Proof of Theorem $\ref{H}$]
Notice that the rescaling of $v$ 
$$v_r(x,t) =v(r \, x, r \, t), \quad \quad r \le 1,$$
satisfies again the hypotheses of Theorem \ref{H} in $\mathcal C_1$ with the constant $\lambda$ replaced by $\lambda_r=\lambda r$.  Proposition \ref{dim_osc} 
applied to $v_r$ implies that 
$$ osc _{\, \, \mathcal C_{1/2}} \, v_r \le (1-c) osc _{\, \, \mathcal C_1} \, v_r$$
which gives (recall that $\mathcal B_{\lambda,r}(y,s)=\mathcal B_{r}(y,s)$ if $y_n=0$),
$$ osc _{\, \, \mathcal B_{r/2}(0,0)}\, v \le (1-c) osc _{\, \, \mathcal B_{r}(0,0)} \, v.$$
Similarly, if $(y,s) \in \overline{\mathcal C}_{1/2} \cap \{x_n=0\}$, then by considering cylinders centered at $(y,s)$ we obtain
\begin{equation}\label{brs}
osc _{\, \, \mathcal B_{r/2}(y,s)} \, v \le (1-c) osc _{\, \, \mathcal B_r(y,s)} \, v, \quad \quad \forall r \le 1/2,
\end{equation} 
which proves the desired oscillation decay on $\{x_n=0\} \cap \overline{\mathcal  C}_{1/2}$. 

If $(y,s) \in \mathcal C_{1/2},$ then \eqref{brs} applied at $((y',0),s)$ implies
$$osc _{\, \, \mathcal B_{\lambda,r/8}(y,s)} \, v \le (1-c) osc _{\, \, \mathcal B_{\lambda,r}(y,s)} \, v, \quad \quad \mbox{if} \quad y_n \le r \le 1/4.$$
In the case when $r < y_n,$ then the inequality above follows from the standard parabolic Harnack inequality 
applied to $v$ in the interior cylinder $\mathcal B_{\lambda,r}(y,s)$.

\smallskip

The boundary version follows in the same way. Precisely, if  $(y,s) \in \overline{\mathcal C}_1 \cap \{x_n=0\}$ then we find
$$osc _{\, \, \mathcal B_{r/2}(y,s) \cap \overline{\mathcal C}_1} \, v \le (1-c) osc _{\, \, \mathcal B_{r}(y,s)\cap \overline{\mathcal C}_1} \, v, \quad \quad \forall r \le 1,$$
 by applying either Proposition \ref{dim_osc} or Lemma \ref{phiB} depending whether or not $\mathcal B_{\lambda,r}(y,s)$ intersects the boundary $\p_D \mathcal C_1$. 

The inequality above can be deduced at all points $(y,s) \in \overline {\mathcal C}_1$ after replacing $r/2$ by $r/8$ on the left hand side. 
Indeed, if $r \ge y_n$ then it follows from the inequality above applied at the point $((y',0),s)$, and if $r< y_n$ 
then we can apply the standard parabolic Harnack inequality or its boundary version since $\mathcal B_{\lambda,r}(y,s)$ does not intersect $\{x_n=0\}$. 
\end{proof}

We conclude the section with the proof of Proposition \ref{Hanw}, that is the Harnack inequality for $w$.

\begin{proof}[Proof of Proposition $\ref{Hanw}$] By Lemma \ref{dim_osc_w} we find that, in terms of $u$, we satisfy again the hypotheses of Proposition \ref{P!} in $\mathcal C_{\lambda/2}$ with $\lambda$ replaced by $
\lambda/2$, $\eps$ replaced by $2(1-c)\eps$, and with $\delta$ the same. The function $a$ stays the same while $b$ is modified by a small constant. Moreover, the property 
$H(\eps^{1/2})$ of $u$ in $\mathcal C_\lambda$ implies that $u$ satisfies property $H(2 \eps^{1/2})$ in $\mathcal C_{\lambda/2}$. We can iterate this result $k$ times as long as the scale parameter of the property $H(2^k \eps^{1/2})$ remains small, universal, and the hypotheses of Lemma \ref{dim_osc_w} hold:
$$2^k \eps^{1/2} \le c'', \quad \quad \delta \le c', \quad \quad 2^k (1-c)^k \eps \le \eps_1(\delta),$$ with $c''$ small, universal. This means that we can iterate $k$ times if
$$ 2^k \eps^{1/2} \le \eps_2(\delta), \quad \quad \delta \le c'.$$
In terms of $w$, we obtain that its oscillation in $\mathcal C_{2^{-k}}$ is bounded by $2 (1-c)^k$ as long as $k$ satisfies the inequality above. On the other hand for the interior balls $\mathcal B_{\lambda,r}$, by \eqref{ihiw}, $w$ satisfies a similar diminishing of oscillation up to scale $r \sim \eps^{1/2}$, and the conclusion follows. 
\end{proof}  

\section{Proof of Proposition \ref{estimate}}

In this section, we prove Proposition \ref{estimate} by using Theorem \ref{H} and the estimates for the one-dimensional problem  which will be proved in Lemma \ref{2D} of the next section. The constants $C$ in this proof depend on $n$ and $K$. 

\

\noindent{\it Proof of Proposition $\ref{estimate}$.} The proof is divided in four steps.

\medskip

{\bf Step 1 - Interior Estimates}. Let $(y,s) \in \mathcal C_{1/2}$. From Theorem \ref{H} we know that $$ osc_{\, \, \mathcal B_r(y,s)} v \le C r ^\alpha, \quad \quad r= y_n.$$
The rescaling
$$ \tilde v (x,t):= v(y+ rx, s + r ^2 \lambda t ),$$
solves in $Q_1 \times (-1,0)$
$$ \tilde v_t = tr (\tilde A(t) D^2 \tilde v), \quad \quad \tilde A(t):=A(s+ r^2 \lambda t).$$
Since $|A'| \leq \lambda^{-1}$, we have $|\tilde A'(t)| \le C,$ and  we find by interior estimates that $|\tilde v_n(0,0)| \le C \;osc_{Q_1\times (-1,0)} \; \tilde v$, from which we deduce
$$|v_n(y,s)| \le C r^{\alpha -1} = C y_n^{\alpha -1}.$$
On the other hand, we prove in appendix that the difference of two viscosity solutions is still a viscosity solution. 
Thus, the estimates for $v$ can be extended to the derivatives of $v$ in the $x_i$ directions, $i=1,\ldots,n-1$. 
Indeed, by applying the interior H\"older estimates to discrete differences in the $x_i$ directions, and iterating this we find that 
$$\|D^k_{x'} v\| \leq C(k) \quad \mbox{in $\mathcal C_{1/2}$}, \quad \forall k \ge 1.$$ 
In particular, using also the estimate for $v_n$ above, we obtain
$$ \|D^2_{x'} v\| \le C, \quad \quad |v_{in}| \leq C x_n^{\alpha-1} \quad \mbox{in $\mathcal C_{1/2}$.}$$ 

\

{\bf Step 2 - Reduction to 1D.} Combining the interior estimates with our assumptions on $\gamma,$ we obtain that when we restrict $v$ to a 
two-dimensional space in which we freeze the $x'$ variable, say for simplicity $x'=0$, then the 
function $v((0,x_n),t)$ solves in the $x_n,t$ variables the equation 
\begin{equation}\label{LPnew}\begin{cases}
v_t = \frac 1 \lambda \{a^{nn}(t) v_{nn} + h(x_n,t)\}&\text{in $\mathcal C_{1},$}\\
v_t =  \gamma_n(t)v_n + f(t) &\text{on $\mathcal F_{1}$,}
\end{cases}
\end{equation}
with
$$|h| \leq C x_n^{\alpha -1}, \quad |f(t)|\leq C,$$
$$h(x_n,t):=\sum_{(i,j)\ne (n,n)}a^{ij}(t) \, \, v_{ij}((0,x_n),t), \quad \quad f(t):= \sum_{i<n}\gamma _i(t) \, \, v_i(0,t).$$
The boundary condition on $\mathcal F_1$ is understood in the viscosity sense. 

Indeed, if a $C^1$ function $\varphi(x_n,t)$ touches $v(0,x_n,t)$ by above/below, say at $(0,0),$ in $\mathcal B_r(0,0) \subset \R^2$, then  
$$\varphi(x_n,t) + \sum_{i<n} v_i(0,0) x_i \pm C|(x,t)|^{1+\alpha}$$
touches $v$ by above/below at the origin in $\mathcal B_r(0,0) \subset \R^{n+1}$. This follows from the $C^\alpha$ continuity of $v_i$, $i<n$, which implies
\begin{equation}\label{v0xnt}
\left| v(x,t) - \left(v(0,x_n,t) + \sum_{i<n} v_i(0,0) x_i \right) \right| \le C |(x,t)|^{1+\alpha}.
\end{equation}

Now, we can use Lemma \ref{2D} a) for $v(0,x_n,t)$, where we establish $C^{1,\alpha}$ estimates for the 1D problem \eqref{LPnew}. We obtain
$$|v((0,x_n),t) -v(0,t) - v_n(0,t)x_n| \le C x_n^{1+\alpha},$$
which together with \eqref{v0xnt} gives
$$\left|v-\left(v(0,t)+ v_n(0,t)x_n + \sum_{i=1}^{n-1}v_i(0,0)x_i\right) \right| \leq C \rho^{1+\alpha} \quad \text{in $\mathcal C_{\rho}.$}$$
This means that 
$$|v-l_{a,b}|  \leq C \rho^{1+\alpha} \quad \text{in $\mathcal C_{\rho},$}$$
with
$$a(t):= (v_1(0,0), \ldots, v_{n-1}(0,0),v_n(0,t)), \quad  b(t) := v(0,t),$$
and
$$b'=\gamma_n(t) a_n +f(t)=\gamma(t) \cdot a + \sum_{i<n}\gamma_i(t)(v_i(0,t)-v_i(0,0)).$$
\

{\bf Step 3 - Modifying the linear approximation.} Next, we modify $a$ and $b$ slightly into $\bar a$, $\bar b$ so that 
$$|v-l_{\bar a, \bar b}|  \leq C \rho^{1+\alpha} \quad \text{in $\mathcal C_{\rho},$}$$
and we also satisfy
\begin{equation}\label{modab}
|\bar a'(t)| \le C \lambda^{-1} \rho^{\alpha-2}, \quad \quad \bar b'=\gamma(t) \cdot \bar a.
\end{equation}
By Lemma \ref{2D} we know that
\begin{equation}\label{ant}
 |a_n(t)-a_n(s)| \le C \lambda^{-\frac \alpha 2}|t-s|^\frac{\alpha}{2},
 \end{equation} and by the H\"older continuity of the $v_i$'s,
\begin{equation}\label{b-g}
|b'-\gamma (t)\cdot a| \le \sum_{i<n}|\gamma_i| |v_i(0,t)-v_i(0,0)| \le C |t|^\alpha.
\end{equation}
Thus, $a_n$ oscillates $C\rho^\alpha$ in an interval of length $\lambda \rho^2$. We define $\bar a$ by averaging $a$ over intervals of this length.
More precisely, let $\eta$ be a standard mollifier in $\R$ with compact support in $[-1,1]$, and $\eta_{\tau}$ denote its rescaling with support of size 
$\tau$. We extend $a_n(t)$ to be constant for $t \ge 0$ and define
$$\bar a_n := a_n * \eta_{\lambda \rho^2},  \quad \quad \bar a_i :=  a_i, \quad i=1,\ldots,n-1.$$
Then \eqref{ant} implies the inequality \eqref{modab} for $\bar a'$ and also
\begin{equation}\label{a-b}
|a-\bar a| \le C \rho^\alpha.
\end{equation}
We define $\bar b(t)$ for $t \le 0$ as
$$\bar b'=\gamma (t) \cdot \bar a, \quad \quad \bar b(0)=b(0).$$
Then, \eqref{b-g}, \eqref{a-b} imply
$$|(\bar b - b)'| \le C \rho^\alpha \quad \Longrightarrow \quad |\bar b - b| \le C \rho^{1+\alpha} \quad \mbox{in $[-\rho,0]$,}$$
and the desired conclusion follows.

\medskip

{\bf Step 4 - Conclusion.} The tangential derivatives $v_i,$ with $i<n$, satisfy the same estimates as $v$. We find from Step 2 applied to $v_i$ that the mixed 
derivatives $v_{in}$ 
must be bounded by a universal bound. This improves the initial estimate in Step 1, which in turn improves the regularity of $f$ and $h$ in Step 2. 
More precisely, by Lemma \ref{2D} we find that $v_{in}$ satisfies the estimate \eqref{w_x}. This holds also for the tangential derivatives of order up to 2. 
Then the functions $h(x,t)$ and $f(t)$ in \eqref{LPnew} satisfy the hypotheses of part b) of Lemma \ref{2D}. This gives that the remaining second 
derivative $v_{nn}$ is 
bounded as well, and \eqref{ant} holds for $\alpha+1$ instead of $\alpha$. Thus we can replace $\alpha$ by $\alpha +1$ in the bound \eqref{modab} 
above, and the proposition is proved.
\qed

\section{Estimates for the 1D case}

In this section, we provide the necessary estimates for solutions to the 1D linear problem. 
The difference with the higher dimensional case is that now, in the 1D case, the H\"older estimates and the subsequent $C^{1,\alpha}$ and $C^{2,\alpha}$ estimates can be iterated 
in parabolic 
cylinders $$\mathcal P_\rho:=(0,\rho) \times (-\rho^2,0],$$
and we can use the standard H\"older parabolic norms with respect to the standard parabolic distance: $d((x,t), (y, s)):= |x-y| + |t-s|^{1/2}$. Following Krylov \cite{K}, we denote the corresponding H\"older spaces with respect to this distance with $C_{x,t}^{k,\alpha}.$

Precisely, we prove the following.

\begin{lem}[1D-Estimates]\label{2D} Assume that $\lambda \le 1$ and $w(x,t)$ is a viscosity solution in $\mathcal C_1 \subset \R^2$ of the equation
\begin{equation}\label{LPnew2D}\begin{cases}
w_t = \frac 1 \lambda \{A(t) w_{xx} + h(x,t)\}&\text{in $\mathcal C_{1}$,}\\
w_t =  \gamma(t) \, w_x + f(t) &\text{on $\mathcal F_{1}$,}
\end{cases}
\end{equation}
with
$$ \|w\|_{L^\infty} \le 1, \quad  K^{-1} \le A(t), \, \gamma(t) \le K, \quad \quad |A'(t)| \le K \lambda^{-1}.$$

a) If 
$$|h| \leq K x^{\alpha -1}, \quad |f(t)|\leq K,$$
then $w \in C^{1,\alpha}$ in the $x$ variable, $w \in C^1$ on $\{x=0\}$, and the free boundary condition is satisfied in the classical sense. 
More precisely, in $\mathcal C_{1/2}$ we have
$$ |w(x,t)- (w(0,t) + x w_x(0,t))| \le C x^{1+\alpha}, \quad \quad \quad |w_x|\le C, $$
and
$$|w(y,t)-w(z,s)| \le C\left(|y-z|^\alpha + \lambda^{-\frac \alpha 2}|t-s|^\frac \alpha 2\right),$$
\begin{equation}\label{w_x}
 |w_x(y,t)-w_x(z,s)| \le C\left(|y-z|^\alpha + \lambda^{-\frac \alpha 2}|t-s|^\frac \alpha 2\right),
\end{equation}
with $C$ depending only on $K$ and $\alpha$.

\

b) If in addition in $\mathcal C_{3/4}$
$$ |h(y,t)-h(z,s)| \le K\left(|y-z|^\alpha + \lambda^{-\frac \alpha 2}|t-s|^\frac \alpha 2\right),$$
$$ |\gamma(t)-\gamma(s)| \le K\lambda^{-\frac \alpha 2}|t-s|^\frac \alpha 2, \quad \quad |f(t)-f(s)| \le K\lambda^{-\frac \alpha 2}|t-s|^\frac \alpha 2,$$
then in $\mathcal C_{1/2}$
\begin{equation}\label{w_x2}
|w_x(0,t)-w_x(0,s)| \le C \lambda^{-\frac {1+\alpha}{ 2}}|t-s|^{\frac{1+ \alpha}{ 2}}, \quad \quad |w_{xx}| \le C.
\end{equation}
\end{lem}

After subtracting $F(t):=\int _0^t f(s)ds$ from $w$ and replacing $h$ by $h - \lambda f(t)$ we may assume that $f \equiv 0$. 
We work with $v(x,t)=w(x,\lambda t)$, and after relabeling $\lambda t$ by $t$ in the arguments of $A$ and $h,$ we obtain 
\begin{equation}\label{LPoneda}\begin{cases}
v_t = A(t) v_{xx} +  h(x,t) &\text{in $(0,1) \times (-\lambda^{-1},0]$,}\\
v_t =  \lambda \gamma(t) \, v_x &\text{on $\{x=0\}$,}
\end{cases}
\end{equation}
with 
\begin{equation}\label{cofLP}
K^{-1} \le A(t), \, \gamma(t) \le K, \quad \quad |A'(t)| \le K, \quad | h | \le K x^{\alpha-1}.
\end{equation}

Lemma \ref{2D} is equivalent to the Lemma \ref{2Dv} below, where we establish the corresponding estimates for $v$ using parabolic scaling. 

\begin{lem}\label{2Dv}
Assume that $v$ is a viscosity solution of \eqref{LPoneda} in $\mathcal P_1$ with $\lambda \le 1$, and coefficients that satisfy \eqref{cofLP}. Then
\begin{equation}\label{vc1a}
\|v\|_{C_{x,t}^{1,\alpha}(\mathcal P_{1/2})} \le C ( \|v\|_{L^\infty(\mathcal P_1)} +1),
\end{equation}
and the free boundary condition is satisfied in the classical sense. If in addition
$$\|h\|_{C_{x,t}^{0,\alpha}}, \quad \|\gamma\|_{C_t^\frac \alpha 2}  \le K,$$
then 
$$\|v\|_{C_{x,t}^{2,\alpha}(\mathcal P_{1/2})} \le C ( \|v\|_{L^\infty(\mathcal P_1)} +1),$$
with $C$ depending only on $n$, $K$ and $\alpha$.

\end{lem}

\begin{proof}

If $v$ solves \eqref{LPoneda} in $\mathcal P_\rho$ then the rescaling $$\tilde v(x,t):=\rho^{-\beta}v(\rho x, \rho^2 t) $$
solves \eqref{LPoneda} in $\mathcal P_1$ with coefficients 
\begin{equation}\label{A(t)}
\tilde A(t)= A(\rho^2t), \quad \tilde h(x,t)=\rho ^{2-\beta} h(\rho x, \rho^2 t), \quad \tilde \lambda = \rho \lambda , \quad \tilde \gamma(t)=\gamma(\rho^2t).
\end{equation}
Notice that the hypotheses on the coefficients are preserved as long as $\beta \le 1+ \alpha$, and moreover $\tilde \lambda \to 0$ as $\rho \to 0$.

\smallskip

We divide the proof in four steps.

\medskip

{\bf Step 1: H\"older estimates.} We show that $$\|v\|_{C_{x,t}^{0,\beta}(\mathcal P_{1/2})} \le C \left(\|v\|_{L^\infty(\mathcal P_1)}+1\right),$$
for some $\beta>0$ small.

Notice that after an initial dilation, we may assume that $\lambda \le \lambda_0$ is small. 
It suffices to prove the following claim.

If $v$ is a viscosity solution of \eqref{LPoneda} then
\begin{equation}\label{oscP}
\, osc_{\, \, \mathcal P_1} \, \, v\le 2 \quad \Longrightarrow \quad osc_{\, \, \mathcal P_{\rho}} \, \, v \le \frac 32, \quad \quad \mbox{with $\rho=c_0$ small, universal.}
\end{equation}
The H\"older estimate is obtained by iterating this claim in parabolic cylinders centered on the $t$ axis, while for the interior parabolic cylinders 
(included in $\{x>0\}$) we can apply directly the diminishing of oscillation for parabolic equations.

In order to prove \eqref{oscP}, we let $g(x,t)$ be the solution to the 1D heat equation on the real-line
\begin{equation}\label{gsubt}
g_t= K^{-1} g_{xx}, \quad g(x,0)=\chi_{(0,\infty)}- \chi_{(-\infty,0)}.
\end{equation}
Notice that for all $t>0$, in $x=0$ we have 
$$g(0,t)=0,  \quad \quad g_x(0,t) \le C t^{-1/2},$$ and 
$$g_t \le 0, \quad \text{for $x>0$}.$$
We want to show that if $|v| \le 1$ in $\mathcal P_1,$ then we can improve the upper bound or lower bound 
by a fixed amount in the interior, depending on the value of $v$ at $(0,-1)$, i.e.
$$\mbox{$|v| \le 1$ in $\mathcal P_1$ and $v(0,-1) \le 0,$ then $v \le 1/2$ in $\mathcal P_\rho$, with $\rho=c_0$.}$$
In $\mathcal P_1$ we compare $v$ with 
$$ G(x,t):=C_1 g(x,t+1) + \frac 14 (t+1)^{1/2} - C_2 x^{1+\alpha}.$$
We choose $C_2$ and then $C_1$ sufficiently large such that $G$ is a classical supersolution to \eqref{LPoneda} and $ G \ge 1$ on the boundary $(0,1] \times \{-1\}$ and $\{1\} \times [-1,0]$, while $G(0,0)=1/4$. Then we find $v \le G$ in $\mathcal P_1,$ which gives the claim \eqref{oscP} by choosing $c_0$ sufficiently small.

\medskip

{\bf Step 2: $C^{1,\alpha}$ estimates.} We show that \eqref{vc1a} holds by first establishing a pointwise $C^{1,\alpha}$ estimate at the origin.

After an initial dilation and after dividing by a large constant, we may assume that $\lambda \le \delta$, $|h| \le \delta x^{\alpha-1}$ for some small $\delta$, and $\|v\|_{L^\infty(\mathcal P_1)}$ is sufficiently small. 

{\it Claim.} If a function $l_0$ (linear in $x$) of the form
\begin{equation}\label{l0}
l_0=a_0 x + b_0(t), \quad \quad b_0'=\lambda \gamma(t) a_0, \quad \quad |a_0| \le 1,
\end{equation} 
approximates $v$ in $\mathcal P_\rho$ to order $1+\alpha$, i.e.
$$|v-l_0| \le \rho^{1+\alpha} \quad \quad \mbox{ in $\mathcal P_\rho$}, \quad \quad \rho\le \delta,$$
then we can approximate $v$ to order $1+\alpha$ in $\mathcal P_{c_1 \rho}$ by a function $l_{1}$ as above, with $|a_1-a_0| \le C \rho^\alpha$, and $c_1$ small universal.
 \texttt{}
Then the claim can be iterated indefinitely by starting with $l_0 \equiv 0$ in $\mathcal P_{\delta}$. 
  
 We prove the claim by compactness. Notice that $v-l_0$ solves \eqref{LPoneda} with a slightly modified $h$ that satisfies $|h| \le \delta x^{\alpha-1} + C \delta$.
 This means that the rescaled error 
 $$\tilde v(x,t):=\rho^{-(1+\alpha)} (v-l_0)(\rho x, \rho^2 t),$$
 satisfies \eqref{LPoneda} with coefficients as in \eqref{A(t)}. Since $\|\tilde v\|_{L^\infty} \le 1$, by Step 1 we know that
$$\|\tilde v\|_{C_{x,t}^{0,\beta}(\mathcal P_{1/2})} \le C.$$
 This means that if we consider a sequence of $\delta_n \to 0$ and corresponding solutions $v_n$ in $\mathcal P_{\rho_n},$ then we can extract a uniformly convergence subsequence of the rescalings $\tilde v_n$ in $\mathcal P_{1/2}$ such that 
 $$ \tilde v_n \to \bar v.$$
 Then the H\"older continuous limit function $\bar v$ is a viscosity solution of 
  \begin{equation*}\begin{cases}
\bar v_t = \bar A \, \bar v_{xx} &\text{in $\mathcal P_{1/2}$,}\\
\bar v_t =  0 &\text{on $\{x=0\}$,}
\end{cases}
\end{equation*}
with $\bar A$ constant. Since $\bar v$ is constant on the boundary $\{x=0\}$, the $C^2$ estimate for the standard heat equation implies
$$|\bar v - (\bar a x + \bar b)| \le C \tau^{2} \le \frac 12 \tau^{1+\alpha}\quad \quad \mbox{in} \quad \mathcal P_\tau, \quad \tau \le c_1.$$
This shows that if $\delta$ is chosen sufficiently small, then the rescaling $\tilde v$ satisfies the inequality above instead of $\bar v$ which implies
$$ |v - (a_1 x + b(t))| \le \frac 3 4  ( \tau \rho )^{1+\alpha} \quad \quad \mbox{in} \quad \mathcal P_{\tau \rho}, \quad \tau=c_1,$$
with 
$$a_1=a_0+ \rho^\alpha \bar a, \quad b(t)=b_0(t) + \rho^{1+\alpha} \bar b.$$
We define $b_1(t)$ so that $l_1$ has the form as in \eqref{l0}, that is $$ b_1'(t)= \lambda \gamma(t) a_1, \quad b_1(0)=b(0).$$ 
Then 
$$|(b_1-b)'| \le C |\bar a \rho^\alpha | \le C \rho^\alpha \quad \Longrightarrow \quad |b_1-b| \le C \rho^\alpha (\tau \rho)^2 \le \frac 1 4  ( \tau \rho)^{1+\alpha} \quad \mbox{in} \quad \mathcal P_{\tau \rho},$$
where we used $\rho \le \delta$ sufficiently small. In conclusion,
$$|v-l_1| \le   ( \tau \rho)^{1+\alpha} \quad \mbox{in} \quad \mathcal P_{\tau \rho}, \quad \quad l_1=a_1x + b_1(t),$$
and the claim is  proved.

\

We remark that the oscillation of $b_0(t)$ which appears in the approximation function $l_0$ in \eqref{l0} is less than $C\rho^2$ in $\mathcal P_\rho$. 
Thus we can modify $b_0$ to be constant in \eqref{l0} and take $l_0$ to be linear, and then adjust the error $\rho^{1+\alpha}$ by $C \rho^{1+\alpha}$. 
This pointwise $C^{1,\alpha}$ estimate can be applied at other points on $\{x=0\}$, 
which combined with interior $C^{1,\alpha}$ estimates for parabolic equations implies the desired conclusion \eqref{vc1a}.

\medskip

{\bf Step 3. Boundary regularity.} We check that $v$ is $C^1$ on $\{x=0\}$ and the boundary condition is satisfied in the classical sense.
 
For this assume by contradiction that there exists a sequence $t_k \to 0^-$ such that 
\begin{equation}\label{1tn}
\frac {1}{t_k}(v(0,t_k)-v(0,0)) < \mu := \lambda \gamma(0) (v_x(0,0) - \eta), \quad \mbox{for some $\eta>0$}.  
\end{equation}
For each $k,$ we look at the contact point where the graph of $v$ is touched by below by a translation of the graph of the classical strict subsolution to \eqref{LPoneda}
$$g(x,t):=v(0,0) + \mu t + x\left(v_x(0,0)- \frac 12 \eta\right) + C x^{1+\alpha},$$
in the domain $D_k:=[0,c(\eta)] \times [t_k,0]$. 

We choose $c(\eta)$ small such that
$g_x(x,t) < v_x(x,t)$ in the domain $D_k$ for all large $k$. 
This implies that the contact point must occur on $D_k \cap \{x=0\}$. 
On the other hand, \eqref{1tn} gives $$v(0,t_k)-v(0,0) > g(0,t_k)-g(0,0)$$ which shows that the contact point is different than $(0,t_k)$ and we reach a contradiction.

\medskip

{\bf Step 4. $C^{2,\alpha}$ estimates.} On $\{x=0\}$ we know that $v_x,\gamma \in C^{\alpha/2}$, 
and the boundary condition implies that $v(0,t) \in C^{1,\alpha/2}$. Now we can apply the standard $C^{2,\alpha}$ Schauder estimates up to the boundary 
for the heat equation. 
\end{proof}

\section{Viscosity solutions for the linear problem}

In this section, we collect some general facts about viscosity solutions for the linear problem \eqref{LP} and establish the existence and uniqueness claim in Proposition  \ref{DiPr} by Perron's method. Similar results for different types of boundary conditions were established by G. Lieberman (see for example \cite{L}). However, we are not aware of an existence result that applies directly to the linear problem \eqref{LP}. Therefore, for completeness we provide the details in this case. 
 
Recall that $v \in C (\mathcal C_1)$ satisfies 
\begin{equation}\label{LPA}\begin{cases}
\lambda v_t \le  tr(A(t) D^2v)&\text{in $\mathcal C_{1},$}\\
v_t \le \gamma(t) \cdot \nabla v &\text{on $\mathcal F_{1}$,}
\end{cases}
\end{equation}
in the viscosity sense if $v$ cannot be touched by above at any point 
$(x_0,t_0) \in \mathcal C_1\cup \mathcal F_1$ in a small neighborhood 
$\mathcal B_{r}(x_0,t_0)$ by a classical strict supersolution $w \in C^2 (\overline {{\mathcal B}_r(x_0,t_0)})$. As usually, this definition is equivalent to the one where we restrict $w$ to belong to the class of quadratic 
polynomials rather than to the class of $C^2$ functions.

Another equivalent way is to say that $v$ is a viscosity subsolution of the parabolic equation in $\mathcal C_1$, 
and a viscosity subsolution of the boundary condition on $\mathcal F_1$. This last condition means that we cannot touch $v$ locally by above at any point 
$(x_0,t_0) \in \mathcal F_1$ by a function $w \in C^1 (\overline {{\mathcal B}_r(x_0,t_0)})$ (or say $w$ is a linear function) that satisfies
$$w_t(x_0,t_0) > \gamma(t_0) \cdot \nabla w(x_0,t_0).$$
The two definitions are the same since, if $w \in C^1$ is as above, and say $(x_0,t_0)=(0,0)$, then a vertical translation of the quadratic polynomial 
$$w(0)+(w_t(0)-\eps)t + (\nabla w(0)+ \eps e_n) \cdot x + M(|x'|^2 - nK^2 x_n^2),$$ 
must touch $v$ by above at some interior point $(x,t) \in \mathcal B_r$. Here $r$ is chosen sufficiently small 
and $M$ large, appropriately, and then the polynomial is a strict supersolution in $\mathcal B_r$.

We state the comparison principle for viscosity solutions.

\begin{lem}\label{comp}
Assume $v_1$ is a viscosity subsolution, and $v_2$ a viscosity supersolution to \eqref{LP} in $\overline{\mathcal C_1}$. 
If $v_1\le v_2$ on $\p_D \mathcal C_1$ then $v_1 \le v_2$ in $\mathcal C_1$.
\end{lem}
\begin{cor}\label{comp1}
The difference of two viscosity solutions of \eqref{LP} is also a viscosity solution of \eqref{LP}.
\end{cor}

We work with the rescaling $w(x,t)=v(x,\lambda t)$.

First we prove a preliminary result on the evolution in time of a Lipschitz ``trace" $w((x',0),t)$ under specific growth assumptions.

\begin{lem}
Assume that $w\leq 1$ satisfies 
\begin{equation}\label{wsub}\begin{cases}
w_t \le \mathcal M^+_K (D^2w) + 1 &\text{in $(Q_1\cap\{x_n>0\}) \times (0,T],$}\\
\frac 1 \lambda w_t \le K w_n^+ -K^{-1} w_n^- + K |\nabla_{x'} w|&\text{on $\{x_n=0\}$,}
\end{cases}
\end{equation}
and
$$ w((x',0),0) \le |x'|^2.$$
Then $$w(0,t) \le C \lambda (t^{1/2} +t) \quad \mbox{for $t \ge 0$,}$$
with $C$ depending on $n$ and $K$.
\end{lem} 

\begin{proof}
We compare $w$ with 
$$G(x,t):= g(x_n,t) + C \lambda (t^{1/2} + t) + |x'|^2  + C (2 x_n - x_n^2),$$
where $g(x_n,t)$ is the solution to the 1D heat equation on the real-line (see \eqref{gsubt})
$$ g_t= K^{-1} g_{nn}, \quad g(x_n,0)=\chi_{(0,\infty)}- \chi_{(-\infty,0)}.$$
It is easy to check that $G$ is a classical supersolution which is above $w$ on the boundary of our domain, and that gives the desired result.
\end{proof}

\begin{lem}\label{nabw}
Assume that $w\le 1$ satisfies \eqref{wsub} in $\mathcal C_1$ and the trace of $w$ on $\{x_n=0\}$ is Lipschitz, i.e.
$$|\nabla_{x'} w| \le 1 \quad \mbox{ on $\{x_n=0$\}}.$$
Then
$$w((x',0),t) \ge w((x',0), 0) - C \lambda^ \frac 2 3 \,| t|^\frac 1 2 \quad \quad  \mbox{if $ x' \in Q'_{1/2}.$}$$
\end{lem} 

\begin{proof}
We prove the inequality for $x'=0$. Since $w$ is Lipschitz the parabola 
$$w(0,t) + C r^2 + r^{-2} |x'|^2 $$ 
is greater than $w((x',0),t),$ with $r$ to be specified later. Now we can apply the previous lemma to the rescaling
$$\tilde w (y,s):= w(ry,t+r^2s) - w(0,t)-C r^2,$$ which solves \eqref{wsub} with $\tilde \lambda=\lambda r$, and obtain that
$$\tilde w (0,s) \le C \tilde \lambda (s^{1/2} + s).$$ 
This gives
$$C (r^2 +  \lambda | t|^ \frac 12 + \lambda r^{-1}|t|) \ge w(0,0) - w(0,t),$$ 
and we choose $r=(\lambda |t|)^{1/3}$ to get
$$w(0,t) \ge w(0,0) - C (\lambda | t|^ \frac 12  + (\lambda |t|)^{2/3}) \ge w(0,0)-C \lambda^ \frac 2 3 \, |t|^\frac 1 2.$$
\end{proof}

\begin{rem}\label{perr}
The proof of Lemma \ref{nabw} shows that we can construct a supersolution $\bar G(x,t)$ in $\mathcal C_1$ such that $\bar G((x',0),-1)=|x'|$, 
$\bar G \ge 1$ on the remaining part of $\p_D \mathcal C_1$, and so that $\bar G(0,t) \le C \lambda ^\frac 2 3 \, |t| ^\frac 12$. 
Similarly, given $\alpha>0$, we can construct a supersolution with $\bar G((x',0),-1)=|x'|^\alpha$, 
$\bar G \ge 1$ on the remaining of $\p_D \mathcal C_1$ and such that $\bar G(0,t) \le C (\lambda |t|)^\beta,$ for some $\beta$ depending on $\alpha$.
\end{rem}

\medskip

We are now ready to prove our main lemma.

\smallskip

{\it Proof of Lemma $\ref{comp}.$} Let $w_i(x,t)=v_i(x,\lambda t)$, $i=1,2$, so that $w_1$ is a subsolution and $w_2$ a  supersolution of
$$\begin{cases}
w_t = tr(A(t) D^2w) &\text{in $\{x_n>0\}$,}\\
\frac 1 \lambda w_t = \gamma(t) \cdot \nabla w&\text{on $\{x_n=0\}$,}
\end{cases}
$$
and we want to show that $w_1$ cannot touch $w_2$ strictly by below at an interior point. Assume by contradiction that this is the case.

The standard viscosity theory of parabolic 
equations implies that the  contact point cannot occur in $\{x_n>0\}$. Below we denote by $C$, $c$ various constants that may depend on $w_i$ and $\lambda$.

%Assume by contradiction that $w_1$ $w_2$ is touched by above locally at a point on $\{x_n=0\}$  by a quadratic polynomial 
%$P$ which is a strict supersolution. 

After a translation and a dilation we may assume that in $\mathcal C_1$
$$w_1 \le  w_2 + \mu t , \quad \quad w_1(0,0)=w_2(0,0)=0,$$ 
for some $\mu>0$ small.
Without loss of generality we may also assume that $w_1/w_2$ has a semiconvex/semiconcave trace in the $x'$ variable, that is
\begin{equation}\label{d2w}
D^2_{x'} w_1 \ge - I,\quad \quad D^2_{x'}w_2 \le I,
\end{equation}
and also
\begin{equation}\label{L^infty-bound-w_i}
\|w_i\|_{L^\infty} \le 1
\end{equation} 
and each $w_i$ solves the parabolic equation in the interior.
This is achieved in the following way. First we replace a subsolution $w$ with the standard regularization using the sup-convolutions in the $x'$ 
variable
$$w_\eps(x,t)= \max_{y} \left\{w(y,t) -\frac {1}{2 \eps}|y'-x'|^2\right\},$$
then we divide $w_\eps$ by a large constant, and in the end we solve the parabolic equation in the interior of $\mathcal C_1$ by keeping the same boundary 
values on the parabolic boundary. All these operations maintain the subsolution property of $w$, and justify the extra assumptions \eqref{d2w}-\eqref{L^infty-bound-w_i}.

 Moreover, after subtracting from each $w_i$ a function of the type $a' \cdot x' + b(t)$ with $\frac{d}{dt}b(t)=\lambda a' \cdot \gamma(t)$  we may assume in addition that 
 \begin{equation}\label{d2w2}
 w_i(0,0)=0, \quad \nabla_{x'}w_i(0,0)=0,
 \end{equation} and the interior 
 parabolic equations have the form
 $$
\partial _t w_i = tr(A(t) D^2w_i) + h(t), \quad \quad \quad |h| \le C.$$
We show that $w_i(0,t)$ are differentiable at the origin in the $t$ variable, and that the derivative of $w_1$ is less than the derivative of $w_2,$ which 
would contradict our hypothesis that $w_1 \le w_2 + \mu t$.

To achieve this we apply Lemma \ref{nabw} several times. By \eqref{d2w}-\eqref{L^infty-bound-w_i}-\eqref{d2w2} and Lemma \ref{nabw} we find that
 \begin{equation}\label{Cr}
 w_1 \ge -  C r \quad \mbox{and} \quad w_2 \le Cr \quad \quad \mbox{ on } \quad \mathcal P_r \cap \{x_n=0\}.
 \end{equation} 
Since $w_1 \le w_2$, we can use the pointwise $C^\alpha$ parabolic estimates at the 
origin and find that, given any $\alpha<1,$ we have
\begin{equation}\label{oscwi}
osc_{\, \mathcal P_r} w_i \le C r^ \alpha \quad \mbox{for all $r>0$.}
\end{equation}
We can iterate this argument, by working with the rescaling
$$\tilde w_1(x,t) = r^{-\alpha} w_1(rx, r^2 t),$$
which satisfies a similar equation with $\tilde \lambda=\lambda r$, and is such that \eqref{d2w}-\eqref{L^infty-bound-w_i}-\eqref{d2w2} hold for $\tilde w_1$. Again by Lemma \ref{nabw} we find
$$ \tilde w_1((x',0),t) \ge - C r^{2/3} \quad \quad \mbox{if $x' \in Q'_{1/2}$},$$
hence we improve the estimate \eqref{Cr} as
\begin{equation}\label{Cr2}
w_1 \ge -  C r^{\alpha + \frac 23} \quad \mbox{ on } \quad \mathcal P_r \cap \{x_n=0\}.
\end{equation}
The same holds for $w_2$ with $\le$ instead of $\ge$ and $C r^{\alpha + \frac 23}$ instead of $-C r^{\alpha + \frac 23}.$\\
This in turn shows that $w_i$ are pointwise $C^{\alpha+\frac 23}$ at the origin. 
 
 We modify again each $w_i$ by subtracting the corresponding function $
 \partial_n w_i(0) 
 x_n + b_i(t)$, with $\frac{d}{dt}b_i=\lambda \gamma_n \partial_n w_i(0)$. Using that $\partial_n(w_1-w_2)(0) \le 0,$ we find that
 the inequality $w_1 \le w_2 + \mu t$ is still valid on $\{x_n=0\},$ while \eqref{oscwi} holds with $r^{\alpha + 2/3}$ instead of $r^\alpha$. 
 The same argument as above implies that \eqref{Cr2} holds again with $r^{\alpha + 4/3}$ instead of $r^{\alpha + 2/3}$. Since $\alpha + 4/3 >2,$ this 
 means that $w_1(0,t) \ge - C |t| ^{1+\beta}$ and $w_2(0,t) \le  C |t| ^{1+\beta}$ for all small $t<0,$ which contradicts $w_1(0,t)\le w_2(0,t) + \mu t$.
 \qed
  
  \medskip
  
We can finally conclude the proof of Proposition  \ref{DiPr}.

\smallskip

{\it Proof of Proposition $\ref{DiPr}.$}
The interior $C^2$ estimates in the $x$ variable and the H\"older estimates up to the boundary were already proved in Proposition \ref{estimate} and 
Theorem \ref{H}.
It remains to prove existence by Perron's method. 

We assume for simplicity that the boundary data $\phi$ is Lipschitz, and the general case follows by approximation.
As usual, we define 
$$v(x,t) := \sup_{w \in \mathcal A} w(x,t),$$
where $\mathcal A$ is the class of continuous subsolutions on $\overline {\mathcal C}_1$ which have boundary data below $\phi$ on 
$\p_D \mathcal C_1$. The conclusion that $v$ solves our problem is easily checked once its continuity has been established.

\smallskip

{\it Claim.} For each $(x_0,t_0) \in \p_D \mathcal C_1$ there exists a subsolution $w_{(x_0,t_0)}$ 
which vanishes at $(x_0,t_0)$, is below the cone $-|(x,t)-(x_0,t_0)|$ on $\p _D \mathcal C_1$ and has a H\"older modulus of continuity at $(x_0,t_0)$.

This can be deduced from the proof of Theorem \ref{H}, where the H\"older continuity at the boundary was achieved using explicit barriers. 
More precisely, as in 
Lemma \ref{phi} and Lemma \ref{phiB}, for all $r \le 1/2$ we can construct a subsolution $\phi_r$ defined in $\mathcal B_{\lambda,r}^{\pm}(x_0,t_0) \cap \mathcal C_1$, where
$$\mathcal B_{\lambda,r}^{\pm}(x_0,t_0):=\{(x,t)| \quad d_\lambda((x,t),(x_0,t_0)) < r\},$$ so that
$$\phi_r=0 \quad \mbox{on} \quad \p \mathcal B_{\lambda,r}^{\pm}(x_0,t_0) \setminus (\p_D \mathcal C_1 \cup \mathcal F_1), \quad \quad 
\phi_r \le 1 \quad \mbox{on} \quad \p \mathcal B_{\lambda,r}^{\pm}(x_0,t_0)  \cap \p_D \mathcal C_1 $$
and
$$\phi_r \ge c_0 \quad \mbox{on} \quad \p \mathcal B_{\lambda,r/2}^{\pm}(x_0,t_0).$$
Then $w_{(x_0,t_0)}$ is obtained by superposing appropriate multiples of $\phi_r$ for a dyadic sequence of $r=2^{-m}$. We omit the details.

\

Using the claim we can construct a subsolution $\underline \phi$ and supersolution $\overline \phi$ which are H\"older continuous on $\p_D \mathcal C_1$ and agree with 
the boundary data $\phi$. Thus we can restrict the class $\mathcal A$ of subsolutions to satisfy 
\begin{equation}\label{undv}
\underline \phi \le w \le 
\overline \phi. 
\end{equation}This shows that the limit $v$ achieves the boundary data $\phi$ continuously. Moreover, using \eqref{undv} we can replace each $w \in \mathcal A$ 
by its maximum among appropriate $x'$ translations
$$\max_{y'}      \, \, \{w(x-(y',0),t) - C |y'|^\alpha\},$$ and remain in the same class. Therefore we may assume that $\mathcal A$ contains only subsolutions which are uniformly H\"older continuous in the 
$x'$ variable. Using this together with Remark \ref{perr}, we find that the trace of $v$ on $\{x_n=0\}$ is locally H\"older continuous in the $x',t$ variables. This means that the solution $\bar v$ to 
the interior parabolic equation in $\mathcal C_1$ with boundary data $v$ is continuous up to the boundary.
By the maximum principle $\bar v \ge w$ for any $w \in \mathcal A$, and it is straightforward to check that $\bar v \in \mathcal A$, hence $v=\bar v$ is continuous in $\overline {\mathcal C}_1$.
\qed

\section*{Acknowledgement.} The authors would like to thank Sandro Salsa and Fausto Ferrari for fruitful conversations on the topic of this paper. N.F. wishes to thank the Department of Mathematics of Columbia University for its warm hospitality.

\end{document}